\documentclass{article}

\usepackage[a4paper,width=150mm,top=25mm,bottom=25mm]{geometry}
\newcommand{\undertitle}[1]{}
\usepackage{amsthm}
\usepackage{titling}
\newtheorem{theorem}{Theorem}[section]

\newtheorem{proposition}[theorem]{Proposition}

\newtheorem{remark}[theorem]{Remark}
\usepackage{mathtools}
\usepackage{xspace} 
\usepackage{mathrsfs}
\usepackage{esint}
\usepackage{algorithm2e}
\usepackage{algorithmic}
\usepackage{graphicx}
\usepackage{tikz}
\usepackage{xcolor}
\usepackage{caption}
\usepackage{subcaption}
\usepackage[justification=centering]{caption} 
\usepackage{bm}

\usepackage{amsmath} 
\usepackage[utf8]{inputenc} 
\usepackage[T1]{fontenc}    
\usepackage{url}            
\usepackage{booktabs}       
\usepackage{amsfonts}       
\usepackage{nicefrac}       
\usepackage[numbers]{natbib}
\usepackage{doi}
\usepackage{mathrsfs}
\def\Qvec{\mathbf{Q}}
\def\Pvec{\mathbf{P}}
\def \Mvec{\mathbf{M}}

\def\Hvec{\mathbf{H}}

\def\d{\mathrm{d}}
\graphicspath{{Figures/}}

\usepackage{hyperref}       
\usepackage{cleveref}

\title{A variational approach to ferronematics with a dimension reduction}

\author{
Shilpa Dutta\thanks{Institute of Mathematics\newline 
~~~Julius-Maximilians-Universität Würzburg, Germany\newline
\href{mailto:shilpa.dutta@uni-wuerzburg.de}{\texttt{shilpa.dutta@uni-wuerzburg.de, shilpa.jumath@gmail.com}}}
}
	
\begin{document}
\maketitle

\begin{abstract}
We present a variational approach to ferronematics in a three dimensional setting. The ferronematic energy functional is described by two established theories: the Landau-de Gennes energy to explain the nematic part, the micromagnetic energy to explain the magnetic part, and coupling energies between them. We explicitly include the nonlocal stray field energy in a bulk setting and the coupling energy accounting for the nematic and stray field interaction. We prove the existence of an energy minimizer for the introduced ferronematic energy functional in a bulk setting. We then provide a reduced local ferronematic energy in a two-dimensional setting via $\Gamma$-convergence.
\end{abstract}
\\[0.5em]
\textbf{\textit{MSCcodes:} Primary 49J27, 49J45; Secondary 76A15, 78A30}

\section{Introduction}
Ferronematics are complex materials that are suspensions of magnetic nanoparticles (MNPs) in a medium of nematic liquid crystals (NLCs). These materials feature spontaneous magnetization (existence of a ferromagnetic phase) even in the absence of an applied magnetic field. The theoretical prediction of such material was made about five decades ago by Brochard and de-Gennes \cite{brochard1970theory}, and after forty years, the experimental validation of that theory was announced in \cite{Mertelj}. Recently, a few mathematical studies \cite{Dalby, maity2021parameter, Canevari, canevari2025formation} have been conducted using a two-dimensional ferronematic model \cite{Bisht2020}.\newline
However, the ferronematic model in \cite{Bisht2020} pertains to Landau-de Gennes (LdG) theory for NLCs (see e.g. \cite{Gennes}) and Landau description of ferrofluid for MNPs (see e.g. \cite{Pleiner}). The model \cite{Bisht2020} is thus a simplified one in the sense that it neglects the presence of the magnetostatic or stray field in the system. However, the non-local stray field may have a significant influence on the defect localization under an activated magnetic field \cite{dutta2025study}. In that two-dimensional ferronematic model, the stray field energy is included using the Gioia and James approximation \cite{Giogia}.\newline\\
In this work, we present a three-dimensional ferronematic model which is based on two variational theories, e.g., the LdG theory for NLCs and the theory of micromagnetics for the spontaneous magnetization. To enable a quick view of the ferronematic energy, which we discuss in detail in the later part, we briefly present the ferronematic energy in a three-dimensional setting as follows. We assume that $\Omega\subset\mathbb R^3$, and as per convention, we denote the state variables due to Landau-de Gennes by $\Qvec$, and micromagnetics by $\Mvec$. The magnetization $\Mvec$ is subjected to an additional equation to account for a magnetic field $\Hvec_{\Mvec}$, known as the magnetostatic field or stray field, and the involved equation is renowned as the stationary Maxwell's equations (cf. \eqref{25}). In addition, we account for the influence of an activated magnetic field $\Hvec_{ext}$. The generalized ferronematic energy, following \cite{dutta2026existence} (up to almost the same notations) in a bulk setting, can be written as
\begin{eqnarray}\label{energy-intro}
\mathcal{E}_B\left(\Qvec, \Mvec\right) = \int_{\Omega} \biggl[\frac{K_1}{2}|\nabla \Qvec|^2 + \frac{K_2}{2}|\nabla \Mvec|^2 + f_B\left(\Qvec, \Mvec\right) + \frac{\mu_0}{2}|\Hvec_{\Mvec}|^2 - \mu_0\Mvec\cdot \Hvec_{ext}\nonumber\\
- \frac{\mu_0\gamma}{2}\Qvec\Mvec\cdot\Mvec - \frac{\chi_1\mu_0}{2}\Qvec\Hvec_{\Mvec}\cdot\Hvec_{\Mvec} - \frac{\chi_1\mu_0}{2}\Qvec\Hvec_{ext}\cdot\Hvec_{ext}\biggr]\;,
\end{eqnarray}
where the bulk potential is given by
\begin{displaymath}
f_B\left(\Qvec, \Mvec\right) = \frac{\alpha_1}{2}\mbox{tr}(\Qvec^2) + \frac{\alpha_2}{3}\mbox{tr}(\Qvec^3) + \frac{\alpha_3}{4}\mbox{tr}(\Qvec^2)^2 + \frac{\alpha_4}{2}|\Mvec|^2 + \frac{\alpha_5}{4}|\Mvec|^4 - \frac{\mu_0\gamma}{2}\Qvec\Mvec\cdot\Mvec.
\end{displaymath}
We then provide an existence theorem for this 3D-ferronematic energy (cf. \eqref{energy-intro}) and establish the 2D-ferronematic energy, proposed in \cite{dutta2025study} via a dimension reduction technique from $\Gamma$-convergence.\newline\\
Before we present our contribution, we discuss the earlier works on the ferronematic model based on \cite{Bisht2020}. In \cite{Dalby}, the authors explore the ferronematic energy in a one-dimensional setting, where they study order reconstruction phenomena and bifurcation analysis. The authors in \cite{maity2021parameter} perform asymptotic analysis with the minimizers of a rescaled two-dimensional ferronematic energy by letting elastic constants tend to zero, and they have presented numerical experiments to support their results. In \cite{Canevari}, the authors have studied a super-dilute regime of 2D ferronematics setting, which refers to the situation when the magnetic energy has a negligible contribution compared to the NLC energy and the potential due to nematic-magnetic interaction is hardly realized. They have shown that the LdG order parameters converge to a canonical harmonic map, while the limit of the magnetization is a singular profile of line defects that associate point defects following a minimal path. Their theoretical results are also complemented with some numerical results. Recently, in \cite{canevari2025formation}, they extended the earlier work on super-dilute regime from the Dirichlet type boundary conditions to the setting that considers Dirichlet boundary condition for $\Qvec$ and Neumann boundary condition for $\Mvec$. All these works deal with ferronematic energy based on \cite{Bisht2020} in a 2D-setting without accounting for the magnetostatic energy. We incorporate such a complex effect in the study and aim to establish a more generalised version. We also suggest \cite{hubert2008magnetic} for a detailed study on magnetostatic or stray field energy and its physical significance.\newline\\
We organize the next part of the paper as follows. \Cref{mathematical framework} introduces the 3D-ferronematic energy. In \Cref{existence results}, we deal with the existence of a minimizer for the introduced 3D-ferronematic energy. We present the reduced ferronematic energy, resembling the energy discussed in \cite{dutta2025study} in \Cref{derivation of the limiting ferronematic energy}. In \Cref{sec:numerical observations}, we present some numerical observations that support the present ferronematic energy over the earlier works. We close with the discussion with a concluding remark in \Cref{conclusions}.

\section{Mathematical framework}\label{mathematical framework}
We begin with assuming that a ferronematic sample fills a domain $\Omega\subset\mathbb R^3$, and the domain $\Omega$ is bounded, simply connected, and Lipschi. Also we assume a set
\begin{eqnarray}
\mathcal{Q}=\bigl\{\Qvec\in\mathbb R^{3\times 3}: \Qvec^T=\Qvec, \operatorname{tr}(\Qvec)=0\bigl\},\label{eq:3d_set}
\end{eqnarray} which is endowed with the inner product $\Qvec\cdot\Pvec = \mbox{tr}\left(\Qvec\Pvec\right)$ and described by the associated norm $|\Qvec|^2 = \mbox{tr}(\Qvec^2)$. The orientational matrix is defined by $\Qvec: \Omega\rightarrow \mathcal{Q}$ and the magnetization is defined by $\Mvec: \Omega\rightarrow \mathbb R^3$. 
The nematic-magnetic coupling energy $- \frac{\mu_0\gamma}{2}\Qvec\Mvec\cdot\Mvec$, that is the interaction between $\Qvec$ and $\Mvec$ is included in the bulk potential $f_B\left(\Qvec, \Mvec\right)$ (see e.g. \cite{Gennes, hubert2008magnetic, Mertelj}) as follows
\begin{align}\label{bulk_couplin}
f_B\left(\Qvec, \Mvec\right) = \frac{\alpha_1}{2}\mbox{tr}(\Qvec^2) + \frac{\alpha_2}{3}\mbox{tr}(\Qvec^3) + \frac{\alpha_3}{4}\mbox{tr}(\Qvec^2)^2 + \frac{\alpha_4}{2}|\Mvec|^2 + \frac{\alpha_5}{4}|\Mvec|^4 - \frac{\mu_0\gamma}{2}\Qvec\Mvec\cdot\Mvec,
\end{align}
where $\alpha_2, \alpha_3 > 0$ are the constants that are independent of the temperature but are dependent on the material properties, while the constant $\alpha_1$ is related to the absolute temperature in a linear setting as follows
$$\alpha_1 = \alpha_0\left(T - T^*\right),$$
where $\alpha_0$ and $T^*$ represent the characteristic temperatures of NLCs. The constant $\mu_0$ stands for the magnetic permeability of the vacuum  as per the standard convention, and $\gamma>0$ is a dimensionless coupling constant that is responsible to measure the interaction between $\Qvec$ and $\Mvec$ \cite{Mertelj}.\\
The magnetization $\Mvec:\Omega\rightarrow\mathbb R^3$ creates a magnetostatic field or stray field $\Hvec_{\Mvec}:\mathbb R^3\rightarrow\mathbb R^3$ \cite{brown1966magnetoelastic}, which satisfies the Maxwell's equations in a static setting as follows
\begin{eqnarray}\label{25}
\textbf{div} \left(\mu_0 \Hvec_{\Mvec} + \chi_{\Omega} \Mvec\right) &= 0 \mbox { in } \mathbb R^3,\nonumber\\
\textbf{curl}~\Hvec_{\Mvec} &= 0 \mbox { in } \mathbb R^3,
\end{eqnarray}
where we extend $\Mvec$ from $\Omega$ to $\mathbb R^3$ by zero and $\chi_{\Omega}: \mathbb R^3\rightarrow \{0,1\}$ is the characteristic function of $\Omega$. For the Maxwell equations (cf. \eqref{25}) in a static setting, we assume (as per the convention of micromagnetics) that electrical conductivity is negligible. In purely magnetic samples, the stray field $\Hvec_{\Mvec}$ can influence the domain observations \cite{hubert2008magnetic}. It may play a significant role in the formation of boundary as well as interior vortices, creation of domain walls and interior walls, etc. \cite{hubert2008magnetic, desimone2004recent}. Moreover, the presence of an external magnetic field $\Hvec_{ext}: \mathbb R^3\rightarrow\mathbb R^3$ influence on the orientation of the magnetization $\Mvec$ and this phenomenon is renowned as the Zeeman effect \cite{hubert2008magnetic}.\newline\\
Thanks to the Poincaré-de Rham lemma \cite{spivak2018calculus}, there exists a scalar potential $\Phi_{\Mvec}: \mathbb R^3\rightarrow\mathbb R$ that satisfies
\begin{eqnarray}
\Hvec_{\Mvec}=-\nabla\Phi_{\Mvec}
\end{eqnarray}
and so, we assume that $\Phi_{\Mvec}: \mathbb R^3\rightarrow\mathbb R$ is a solution of the following elliptic problem
\begin{eqnarray}
\begin{aligned}
\Delta\Phi_{\Mvec} = \nabla\cdot\Mvec \mbox{ on } \mathbb R^3.\\
\end{aligned}
\end{eqnarray}
Therefore, accounting these possible energy contributions in a ferronematic sample, the energy functional, following \cite{dutta2026existence}, in a bulk setting can be written as
\begin{eqnarray}\label{bulk_energy}
\mathcal{E}_B\left(\Qvec, \Mvec\right) = \int_{\Omega}\biggl[ \frac{K_1}{2}|\nabla \Qvec|^2 + \frac{K_2}{2}|\nabla \Mvec|^2 + f_B\left(\Qvec, \Mvec\right) + \frac{\mu_0}{2}|\Hvec_{\Mvec}|^2 - \mu_0\Mvec\cdot \Hvec_{ext}\nonumber\\
- \frac{\mu_0\gamma}{2}\Qvec\Mvec\cdot\Mvec - \frac{\chi_1\mu_0}{2}\Qvec\Hvec_{\Mvec}\cdot\Hvec_{\Mvec} - \frac{\chi_1\mu_0}{2}\Qvec\Hvec_{ext}\cdot\Hvec_{ext}\biggr]\;.
\end{eqnarray}
The energy functional $\mathcal{E}_B\left(\Qvec, \Mvec\right)$ consists of the following energy contributions:
\begin{itemize}
\item As already mentioned, the nematic part of the energy functional $\mathcal{E}_B\left(\Qvec, \Mvec\right)$ is explained by the Landau-de Gennes theory \cite{Gennes}. The energy density $\frac{K_1}{2}|\nabla\Qvec|^2$ is known as the Landau-de Gennes (LdG) energy density assuming the one-constant approximation, where $K_1>0$ is called the elastic constant. The part of the energy in the bulk potential $f_B\left(\Qvec, \Mvec\right)$, which measures the nematic ordering, generally given by
\begin{eqnarray}
f_B\left(\Qvec\right)=\frac{\alpha_1}{2}\mbox{tr}(\Qvec^2) + \frac{\alpha_2}{3}\mbox{tr}(\Qvec^3) + \frac{\alpha_3}{4}\mbox{tr}(\Qvec^2)^2
\end{eqnarray}
with $\alpha_3>0$ and $\alpha_1<0$ as established in \cite{Gennes}.
\item The magnetic part of the total energy is explained by the theory of micromagnetics \cite{brown1966magnetoelastic}. The energy density $\frac{K_2}{2}|\nabla\Mvec|^2$ is known as the exchange contribution and $K_2>0$ is the exchange constant. This exchange energy density is a penalized contribution, which measures the deviation of magnetic orientation from the equilibrium directions. Similarly, the deviation of magnetic saturation is accounted by the penalty term $\frac{\alpha_4}{2}|\Mvec|^2 + \frac{\alpha_5}{4}|\Mvec|^4$, where $\alpha_4<0$ and $\alpha_5>0$ are the constants, which explain magnetic phase transitions \cite{hubert2008magnetic, Pleiner}. To complete the micromagnetic aspects, we also incorporate the stray field or magnetostatic energy density by $\frac{\mu_0}{2}|\Hvec_{\Mvec}|^2$ and the Zeeman energy density by $-\mu_0\Mvec\cdot\Hvec_{ext}$ \cite{hubert2008magnetic}.
\item Furthermore, the generalized energy $\mathcal{E}_B\left(\Qvec, \Mvec\right)$ incorporates the coupling energy densities $- \frac{\mu_0\gamma}{2}\Qvec\Mvec\cdot\Mvec$ due to interaction between $\Qvec$ and $\Mvec$ with coupling constant $\gamma>0$, $-\frac{\chi_1\mu_0}{2}\Qvec\Hvec_{\Mvec}\cdot\Hvec_{\Mvec}$ that is responsible for the impacts of $\Hvec_{\Mvec}$ on $\Qvec$. The nematic-stray field coupling contribution follows the same structure as the nematic-magnetic coupling term. Finally, it is the term $-\frac{\chi_1\mu_0}{2}\Qvec\Hvec_{ext}\cdot\Hvec_{ext}$ that accounts the effect of $\Hvec_{ext}$ on $\Qvec$. Here $\chi_1$ stands for the magnetic susceptibility \cite{Mertelj}.
\end{itemize}
\section{Existence results}\label{existence results}
To present the existence result, we define the admissible space as follows:
\begin{displaymath}
\mathcal{Q}_{B} = \Big \{ \Qvec \in W^{1,2} \left(\Omega, \mathcal{Q}\right):  \Qvec|_{\partial \Omega} = \Qvec_{bd} \mbox{ in a sense of trace }\Big \},
\end{displaymath}
\begin{displaymath}
\mathcal{M}_{B} = \Big \{ \Mvec \in W^{1,2} \left( \Omega, \mathbb R^3 \right) : \Mvec|_{\partial \Omega} = \Mvec_{bd} \mbox{ in a sense of trace }\Big \},
\end{displaymath}
\begin{displaymath}
\mathcal{Q}_{B}\times\mathcal{M}_{B} = \Big \{ \left(\Qvec, \Mvec\right) : \Qvec \in \mathcal{Q}_{B} \mbox{ and } \Mvec \in \mathcal{M}_{B} \Big \},
\end{displaymath}
where $\Qvec_{bd}: \partial\Omega\rightarrow\mathcal{Q}$ and $\Mvec_{bd}: \partial\Omega\rightarrow\mathbb R^3$ are some given mappings.\newline\\
As a first key step, we state the existence theorem in the context of energy minimization of energy functional (cf. \eqref{bulk_energy}), below.
\begin{theorem}[Existence of energy minimizers]\label{thm:1}
Let $\Omega\subset\mathbb R^3$ be a bounded domain such that $\partial\Omega$ is Lipschitz and $\mathcal{Q}\times\mathcal{M}\ne\emptyset$. Also, we assume that $\Hvec_{ext}\in C\left(\overline{\Omega}; \mathbb R^3\right)$. Then the functional $\mathcal{E}_{B}\left(\Qvec, \Mvec\right)$ (cf. \eqref{bulk_energy}) has a minimum on the space $\mathcal{Q}_B\times\mathcal{M}_B$.
\end{theorem}
\begin{remark}
We note that the admissible space $\mathcal{Q}_B\times\mathcal{M}_B$ is nonempty, that can be instantly concluded from an example provided in \cite{bethuel1994ginzburg}.
\end{remark}
\begin{remark}
The existence theorem (cf. \Cref{thm:1}) holds true with weak regularity assumption of $\Hvec_{ext}$. However for simplicity, we assume that it is continuous up to the boundary $\partial\Omega$.
\end{remark}
The first essential requirement to minimize the energy functional is the lower bound of the functional. As this lower bound cannot be instantly guaranteed, we first define a penalized energy functional in order to prove our main existence theorem, where the idea is to show that minimizers of this penalized energy functional converge to the minimizer of the actual energy functional (cf. \eqref{bulk_energy}). To this aim, we assume the penalized energy functional as follows
\begin{eqnarray}\label{penalized_energy}
\mathcal{E}^{\varepsilon}_B\left(\Qvec, \Mvec\right) = \int_{\Omega} \biggl[\frac{K_1}{2}|\nabla \Qvec|^2 + \frac{K_2}{2}|\nabla \Mvec|^2 + f_B\left(\Qvec, \Mvec\right) + \frac{\mu_0}{2}|\Hvec_{\Mvec}|^2 - \mu_0\Mvec\cdot \Hvec_{ext}\nonumber\\
- \frac{\chi_1\mu_0}{2}\Qvec\Hvec_{\Mvec}\cdot\Hvec_{\Mvec} - \frac{\chi_1\mu_0}{2}\Qvec\Hvec_{ext}\cdot\Hvec_{ext}+\frac{\varepsilon\mu_0}{2}|\Hvec_{\Mvec}|^4\bigg]\;, 
\end{eqnarray}
where
\begin{eqnarray}
f_B\left(\Qvec, \Mvec\right) = \frac{\alpha_1}{2}\mbox{tr}(\Qvec^2) + \frac{\alpha_2}{3}\mbox{tr}(\Qvec^3) + \frac{\alpha_3}{4}\mbox{tr}(\Qvec^2)^2 + \frac{\alpha_4}{2}|\Mvec|^2 + \frac{\alpha_5}{4}|\Mvec|^4 - \frac{\mu_0\gamma}{2}\Qvec\Mvec\cdot\Mvec.
\end{eqnarray}
We first show that the penalized energy functional $\mathcal{E}^{\varepsilon}_B\left(\Qvec, \Mvec\right)$ is bounded from below. To this aim, we let
\begin{eqnarray}
\mathcal{I}_1&=&  - \frac{\mu_0\gamma}{2}\Qvec\Mvec\cdot\Mvec\nonumber\\
&=& - \frac{\mu_0\gamma}{2}\sum_{i,j=1}^3 Q_{ij}M_iM_j\nonumber\\
&=& - \frac{\mu_0\gamma}{2}\left(Q_{11}M_1^2+Q_{22}M_2^2+Q_{33}M_3^2+2Q_{12}M_1M_2+2Q_{13}M_1M_3+2Q_{23}M_2M_3\right)\nonumber\\
&\ge& - \frac{\mu_0\gamma}{2}\left(|Q_{11}|M_1^2+|Q_{22}|M_2^2+|Q_{33}|M_3^2+|Q_{12}|\left(M_1^2+M_2^2\right)\right)\nonumber\\
& & - \frac{\mu_0\gamma}{2}\left(|Q_{13}|\left(M_1^2+M_3^2\right)+|Q_{23}|\left(M_2^2+M_3^2\right)\right)\nonumber\\
&\ge& - \frac{\mu_0\gamma}{2}\left(|Q_{11}|+|Q_{22}|+|Q_{33}|+|Q_{12}|+|Q_{13}|+|Q_{23}|\right)|\Mvec|^2\nonumber\\
&\ge& - \frac{\mu_0\gamma}{2}\left(\epsilon_1\left(|Q_{11}|+|Q_{22}|+|Q_{33}|+|Q_{12}|+|Q_{13}|+|Q_{23}|\right)^2+\frac{1}{\epsilon_1}|\Mvec|^4\right)\nonumber\\
&\ge& - \frac{\mu_0\gamma}{2}\left(\epsilon_1A|\Qvec|^2+\frac{1}{\epsilon_1}|\Mvec|^4\right),\label{lb:1}
\end{eqnarray}
\begin{eqnarray}
\mathcal{I}_2&=&  - \frac{\mu_0\chi_1}{2}\Qvec\Hvec_\Mvec\cdot\Hvec_\Mvec\nonumber\\
&\ge&  - \frac{\mu_0\chi_1}{2}\left(\epsilon_2\left(|Q_{11}|+|Q_{22}|+|Q_{33}|+|Q_{12}|+|Q_{13}|+|Q_{23}|\right)^2+\frac{1}{\epsilon_2}|\Hvec_\Mvec|^4\right)\nonumber\\
&\ge& - \frac{\mu_0\gamma}{2}\left(\epsilon_2A|\Qvec|^2+\frac{1}{\epsilon_2}|\Hvec_\Mvec|^4\right),\label{lb:2}
\end{eqnarray}
\begin{eqnarray}
\mathcal{I}_3&=&  - \frac{\mu_0\chi_1}{2}\Qvec\Hvec_{ext}\cdot\Hvec_{ext}\nonumber\\
&\ge&  - \frac{\mu_0\chi_1}{2}\left(\epsilon_3\left(|Q_{11}|+|Q_{22}|+|Q_{33}|+|Q_{12}|+|Q_{13}|+|Q_{23}|\right)^2+\frac{1}{\epsilon_3}|\Hvec_{ext}|^4\right)\nonumber\\
&\ge&  - \frac{\mu_0\chi_1}{2}\left(\epsilon_3A|\Qvec|^2+\frac{1}{\epsilon_3}|\Hvec_{ext}|^4\right),\label{lb:3}
\end{eqnarray}
\begin{eqnarray}
\mathcal{I}_4&=& -\mu_0\Mvec\cdot\Hvec_{ext}\nonumber\\
&\ge&-\mu_0\left(\epsilon_4|\Mvec|^2+\frac{1}{\epsilon_4}|\Hvec_{ext}|^4\right).\label{lb:4}
\end{eqnarray}
In the above $\epsilon_1, \epsilon_2, \epsilon_3, \epsilon_4>0$ are arbitrary constants due to Young's inequality.\newline\\
Assembling \eqref{lb:1}, \eqref{lb:2}, \eqref{lb:3}, and \eqref{lb:4}, we have
\begin{eqnarray}
& &\frac{\alpha_1}{2}\mbox{tr}(\Qvec^2) + \frac{\alpha_2}{3}\mbox{tr}(\Qvec^3) + \frac{\alpha_3}{4}\mbox{tr}(\Qvec^2)^2 + \frac{\alpha_4}{2}|\Mvec|^2 + \frac{\alpha_5}{4}|\Mvec|^4 - \frac{\mu_0\gamma}{2}\Qvec\Mvec\cdot\Mvec\nonumber\\
& &~~~~~- \mu_0\Mvec\cdot \Hvec_{ext}+\frac{\varepsilon\mu_0}{2}|\Hvec_{\Mvec}|^4- \frac{\mu_0\chi_1}{2}\Qvec\Hvec_\Mvec\cdot\Hvec_\Mvec - \frac{\mu_0\chi_1}{2}\Qvec\Hvec_{ext}\cdot\Hvec_{ext}\nonumber\\
&\ge& \frac{\alpha_1}{2}\mbox{tr}(\Qvec^2) + \frac{\alpha_2}{3}\mbox{tr}(\Qvec^3) + \frac{\alpha_3}{4}\mbox{tr}(\Qvec^2)^2 + \frac{\alpha_4}{2}|\Mvec|^2 + \frac{\alpha_5}{4}|\Mvec|^4 - \frac{\mu_0\gamma}{2}\left(\epsilon_1A_2|\Qvec|^2+\frac{1}{\epsilon_1}|\Mvec|^4\right)\nonumber\\
& &~~~~~- \frac{\mu_0\gamma}{2}\left(\frac{1}{\epsilon_2}A_3|\Qvec|^2+\epsilon_2|\Hvec_\Mvec|^4\right) - \frac{\mu_0\chi_1}{2}\left(\epsilon_3A_4|\Qvec|^2+\frac{1}{\epsilon_3}|\Hvec_{ext}|^4\right)\nonumber\\
& &~~~-\mu_0\left(\epsilon_4|\Mvec|^2+\frac{1}{\epsilon_4}|\Hvec_{ext}|^4\right)+\frac{\varepsilon\mu_0}{2}|\Hvec_{\Mvec}|^4\nonumber\\
&\ge& \frac{\alpha_1}{2}|\Qvec|^2+\frac{\alpha_2}{3}A_1|\Qvec|^3+\frac{\alpha_3}{4}|\Qvec|^4 + \frac{\alpha_4}{2}|\Mvec|^2 + \frac{\alpha_5}{4}|\Mvec|^4 - \frac{\mu_0\gamma}{2}\left(\epsilon_1A_2|\Qvec|^2+\frac{1}{\epsilon_1}|\Mvec|^4\right)\nonumber\\
& &~~~~~- \frac{\mu_0\gamma}{2}\left(\frac{1}{\epsilon_2}A_3|\Qvec|^2+\epsilon_2|\Hvec_\Mvec|^4\right) - \frac{\mu_0\chi_1}{2}\left(\epsilon_3A_4|\Qvec|^2+\frac{1}{\epsilon_3}|\Hvec_{ext}|^4\right)\nonumber\\
& &~~~~~~-\mu_0\left(\epsilon_4|\Mvec|^2+\frac{1}{\epsilon_4}|\Hvec_{ext}|^4\right)+\frac{\varepsilon\mu_0}{2}|\Hvec_{\Mvec}|^4\nonumber\\ 
&\ge&\frac{\alpha_3}{4}|\Qvec|^4+\frac{\alpha_2}{3}A_1|\Qvec|^3+\left(\frac{\alpha_1}{2}-\frac{\epsilon_1A_2\mu_0\gamma}{2} - \frac{A_3\mu_0\gamma}{2\epsilon_2}\right)|\Qvec|^2+\left(\frac{\alpha_5}{4}-\frac{\mu_0\gamma}{2\epsilon_1}\right)|\Mvec|^4\nonumber\\
& &+\left(\frac{\alpha_4}{2}-\mu_0\epsilon_4\right)|\Mvec|^2+\left(\frac{\varepsilon\mu_0}{2}-\frac{\mu_0\gamma\epsilon_2}{2}\right)|\Hvec_{\Mvec}|^4-\left(\frac{\mu_0\chi_1}{2\epsilon_3}-\frac{\mu_0}{\epsilon_4} \right)|\Hvec_{ext}|^4.\label{lb:5}
\end{eqnarray}
We pick the constants due to Young's inequality as $\epsilon_1>\frac{2\mu_0\gamma}{\alpha_5}>0$ and $\frac{\varepsilon}{\gamma}>\epsilon_2>0$ in \eqref{lb:5} so that the penalized energy functional (cf. \eqref{penalized_energy}) is bounded from below by the property of quadratic functions. 
Then from \eqref{lb:5}, it follows that the energy functional \eqref{penalized_energy} is bounded below for any $\varepsilon>0$. The remaining requirements are compactness and lower-semicontinuity of the energy functional $\mathcal{E}^{\varepsilon}_B\left(\Qvec, \Mvec\right)$, which is bounded from below.\newline\\
We begin with the compactness argument in order to apply the direct method in the calculus of variations, below.
\begin{proposition}\label{prop:comp}
Let $\Omega\subset\mathbb R^3$ be a bounded domain such that $\partial\Omega$ is Lipschitz and $\mathcal{Q}_B\times\mathcal{M}_B\ne\emptyset$. Moreover, we assume that $\Hvec_{ext}\in C\left(\overline{\Omega}; \mathbb R^3\right)$. Then the compactness argument can be stated as follows:\\
if $\mathcal{E}^{\varepsilon}_B\left(\Qvec_k, \Mvec_k\right)\le\Upsilon$ for a sequence $\left(\Qvec_k, \Mvec_k\right)\in\mathcal{Q}_B\times\mathcal{M}_B$ and some $\Upsilon\in\mathbb R$ independent of $k$, then $\left(\Qvec_k, \Mvec_k\right)$ has a weakly converging subsequence in $\mathcal{Q}_B\times\mathcal{M}_B$.
\end{proposition}
\begin{proof}
Assume that $\left(\Qvec_k, \Mvec_k\right)_k \subseteq \mathcal{Q}_B \times \mathcal{M}_B$ is an minimizing sequence of the functional $\mathcal{E}^{\varepsilon}_B\left(\Qvec, \Mvec\right)$ such that $\mathcal{E}^{\varepsilon}_B\left(\Qvec, \Mvec\right)\le\Upsilon$ for some positive $\Upsilon\in\mathbb{R}$ and all $k\in \mathbb{N}$. Due to \eqref{lb:5}, we have
\begin{eqnarray}
\int_{\Omega} \frac{K_1}{2}|\nabla \Qvec_k|^2 + \frac{K_2}{2}|\nabla \Mvec_k|^2 + \frac{\mu_0}{2}|\Hvec_{\Mvec_k}|^2\le\Upsilon
\end{eqnarray}
uniformly in $k\in\mathbb{N}$. Thanks to the Banach-Alaoglu theorem, up to a subsequence (without reindexing), we have
\begin{eqnarray}
\Qvec_k&\rightharpoonup&\Qvec \mbox{ weakly in } W^{1,2}\left(\Omega; \mathbb R^{3\times 3}\right),\label{eq:Q-weak}\\
\Mvec_k&\rightharpoonup&\Mvec \mbox{ weakly in } W^{1,2}\left(\Omega; \mathbb R^{3}\right).\label{eq:M-weak}
\end{eqnarray}
By the Kondrachov compact embedding [see e.g. \cite{ciarlet2021mathematical}], it holds that
\begin{eqnarray}
\Qvec_k&\rightarrow&\Qvec \mbox{ strongly in } L^2\left(\Omega; \mathbb R^{3\times 3}\right), \label{eq:Q-strng}\\
\Mvec_k&\rightarrow&\Mvec \mbox{ strongly in } L^2\left(\Omega; \mathbb R^{3}\right). \label{eq:M-strng}
\end{eqnarray}
Finally, thanks to compactness property of trace operator $\gamma:W^{1,2}\rightarrow L^2$ [see e.g. \cite{ciarlet2021mathematical}], up to a subsequence, we have
\begin{eqnarray}
\gamma\left(\Qvec_k\right)\rightarrow~\gamma\left(\Qvec\right)\mbox{ a.e. in }\partial\Omega\mbox{ as }k\rightarrow\infty,\label{eq:ex7}\\
\gamma\left(\Mvec_k\right)\rightarrow~\gamma\left(\Mvec\right)\mbox{ a.e. in }\partial\Omega\mbox{ as }k\rightarrow\infty.\label{eq:ex8}  
\end{eqnarray}
\end{proof}
Next we treat the stationary Maxwell's equation (cf. \eqref{25}). We refer to \cite{Desimmone, MK, dutta2026existence} for a similar explanation; however, here we present a self-contained explanation that fits in the present setting.
\begin{proposition}\label{lem-str}
Let $\chi_{\Omega_k} \textbf{M}_k \rightarrow \chi_{\Omega} \textbf{M}\mbox{ strongly in } L^{2}\left(\mathbb R^3; \mathbb R^3 \right)$ and $\textbf{H}_{\textbf{M}_k}$ be a  solution of Maxwell equation corresponding to $\chi_{\Omega_k}\textbf{M}_k$. Then $\textbf{H}_{\textbf{M}_k} \rightharpoonup \textbf{H}_{\textbf{M}}$ weakly in $L^{2}\left(\mathbb R^3; \mathbb R^3 \right)$, where $\textbf{H}_{\textbf{M}}$ is a solution of stationary Maxwell's equations \eqref{25} corresponding to $\chi_{\Omega}\textbf{M}$, and $||\Hvec_{\Mvec_k}||_{L^2\left(\mathbb R^3; \mathbb R^3\right)}\rightarrow ||\Hvec_{\Mvec}||_{L^2\left(\mathbb R^3; \mathbb R^3\right)}$ as $k\rightarrow\infty$.
\end{proposition}
\begin{proof}
Assuming energy functional $\mathcal{E}^{\varepsilon}_B\left(\Qvec, \Mvec\right)<\infty$ in $\mathcal{Q}_B\times\mathcal{M}_B$, the curl-free property of $\Hvec_{\Mvec}$ can be understood at least in distributional sense.
Thus we regard $\textbf{H}_{\textbf{M}}$ as a member of the function space
\begin{eqnarray}
\mathcal{H} = \{\textbf{H}\in L^2\left(\mathbb R^3; \mathbb R^3\right): \mbox{curl}\textbf{H} = 0 \mbox{ in distributional sense} \}   
\end{eqnarray}
and magnetostatic field $\textbf{H}_{\textbf{M}}$ satisfies
\begin{eqnarray}
\int_{\mathbb R^3}\textbf{H}_{\textbf{M}}\cdot\bm{\psi} dx = -\int_{\mathbb R^3}\chi_{\Omega}\textbf{M}\cdot\bm{\psi} dx ~~\forall\bm{\psi}\in \mathcal{H}.\label{eq:17'}   \end{eqnarray}
Since $\textbf{H}\in L^2\left(\mathbb R^3; \mathbb R^3\right)$, the partial derivative of $\textbf{H}$ can be understood in the sense of distribution and hence by Poincaré-de Rham lemma (see e.g.\ \cite{spivak2018calculus}), it holds that $\nabla\times \textbf{H} = 0$ in a distributional sense iff $\textbf{H} = -\nabla V$ for some $V\in H^1\left(\mathbb R^3\right)$. Then \eqref{eq:17'} can be written as
\begin{eqnarray}\label{18}
\int_{\mathbb R^3}\nabla \textbf{U}_{\textbf{M}}\cdot\nabla Vdx = \int_{\mathbb R^3}\chi_{\Omega}\textbf{M}\cdot \nabla Vdx ~~\forall~V\in H^1\left(\mathbb R^3\right),
\end{eqnarray}
where $\textbf{H}_{\textbf{M}} = -\nabla U_{\textbf{M}}$ with $U_{\textbf{M}}\in H^1\left(\mathbb R^3\right)$.\\
\\
We define a bilinear form $B[U_{\textbf{M}}, V]$ in the following way
\begin{eqnarray}\label{19}
B[U_{\textbf{M}}, V] = \int_{\mathbb R^3}\nabla U_{\textbf{M}}\cdot\nabla Vdx = \int_{\mathbb R^3}\chi_{\Omega}\textbf{M}\cdot \nabla Vdx~~\forall~U_{\textbf{M}}, V\in H^1\left(\mathbb R^3\right),
\end{eqnarray}
for which
$$|B\left(U_{\textbf{M}}, V\right)|\le \mu_0||U_{\textbf{M}}||_{H^1}||V||_{H^1}~~\forall~~U_{\textbf{M}}, V\in H^1\left(\mathbb R^3\right)$$
and
$$B\left(U_{\textbf{M}}, U_{\textbf{M}}\right)\ge \mu_0||U_{\textbf{M}}||^2.$$
Thanks to Lax-Milgram theorem (see e.g.\ \cite{evans2022partial}), there exists $U_{\textbf{M}}\in H^1\left(\mathbb R^3\right)$ uniquely and $U_{\textbf{M}}$ satisfies
\begin{eqnarray}\label{20}
B[U_{\textbf{M}}, V] = \int_{\mathbb R^3}\nabla U_{\textbf{M}}\cdot\nabla V dx = \int_{\mathbb R^3}\chi_{\Omega}\textbf{M}\cdot \nabla Vdx~~\forall V\in H^1\left(\mathbb R^3\right).
\end{eqnarray}
Thus, \eqref{25} possesses a unique solution $\textbf{H}_{\textbf{M}}\in L^2\left(\mathbb R^3; \mathbb R^3\right)$ corresponding to $\textbf{M}\in L^2\left(\mathbb R^3; \mathbb R^3\right)$.\\
The remaining steps are to follow.
By the lower bound \eqref{lb:5} and the uniform upper bound of the energy functional \eqref{penalized_energy}, we deduce that
\begin{eqnarray}
\sup_{k\in \mathbb N}\int_{\mathbb R^3}|\textbf{H}_{\textbf{M}_k}|^2dx<\infty,\label{eq:bound-str}
\end{eqnarray}
and as a consequence, up to a subsequence (without re-indexing), it holds that 
\begin{eqnarray}
\textbf{H}_{\textbf{M}_k}\rightharpoonup \Lambda \mbox{ weakly in } L^2\left(\mathbb R^3; \mathbb R^3\right).\label{eq:weak-str}
\end{eqnarray}
The uniqueness of the solution of \eqref{25} and the linearity of \eqref{25} conclude that $\Lambda=\textbf{H}_{\textbf{M}}$. Since $\textbf{H}_{\textbf{M}}\in L^2\left(\mathbb R^3; \mathbb R^3\right)$, $\nabla\times \textbf{H}_{\textbf{M}}= 0$ in the sense of distribution, immediately follows. Furthermore, plugging in $V=U_{\textbf{M}}$ in \eqref{20}, we can write
\begin{eqnarray}
\int_{\mathbb R^3}|\nabla U_{\textbf{M}_k}|^2dx&=&\int_{\mathbb R^3}\chi_{\Omega}\textbf{M}_k\cdot\nabla U_{\textbf{M}_k}dx\nonumber\\
&=& \int_{\mathbb R^3}\chi_{\Omega}\left(\textbf{M}_k-\textbf{M}\right)\cdot\nabla U_{\textbf{M}_k}dx+\int_{\mathbb R^3}\chi_{\Omega}\textbf{M}\cdot\nabla U_{\textbf{M}_k}dx.\label{eq:strg-str}
\end{eqnarray}
Thanks to the H\"older inequality, it holds that
\begin{eqnarray}
& &\bigg|\int_{\mathbb R^3}\chi_{\Omega}\left(\textbf{M}_k-\textbf{M}\right)\cdot\nabla U_{\textbf{M}_k}dx\bigg|\nonumber\\
&\le&\int_{\mathbb R^3}\bigg|\chi_{\Omega}\left(\textbf{M}_k-\textbf{M}\right)\cdot\nabla U_{\textbf{M}_k}\bigg|dx\nonumber\\
&\underset{\mbox{H\"older}}{\le}& ||\chi_{\Omega}\left(\textbf{M}_k-\textbf{M}\right)||_{L^2}||\nabla U_{\textbf{M}_{k}}||_{L^2}\rightarrow 0 \mbox{ as } k\rightarrow\infty.\label{eq:Hol}
\end{eqnarray}
\noindent The last convergence follows due to the strong convergence \eqref{str-M} and the uniform bound \eqref{eq:bound-str}. As $\nabla U_{\textbf{M}_{k}}\rightharpoonup \nabla U_{\textbf{M}}$ weakly in $L^2\left(\mathbb R^3; \mathbb  R^3\right)$ and $\chi_{\Omega}\Mvec\in L^2\left(\mathbb R^3; \mathbb R^3\right)$, by the definition of weak-convergence, it holds that
\begin{eqnarray}
\int_{\mathbb R^3}\textbf{M}\cdot\nabla U_{\textbf{M}_k}dx = \int_{\mathbb R^3}\textbf{M}\cdot\nabla U_{\textbf{M}}dx.\label{eq:weak}
\end{eqnarray}
By \eqref{eq:Hol} and \eqref{eq:weak}, it is immediate from \eqref{eq:strg-str} that $||\Hvec_{\Mvec_k}||_{L^2\left(\mathbb R^3; \mathbb R^3\right)}\rightarrow ||\Hvec_{\Mvec}||_{L^2\left(\mathbb R^3; \mathbb R^3\right)}$ as $k\rightarrow\infty$.
\end{proof}
Now we complete the lower-semicontinuity argument as another requirement for the application of the direct method in the calculus of variations.
\begin{proposition}\label{prop:wls}
Let $\Omega\subset\mathbb R^3$ be a bounded Lipschitz domain. Then the liminf inequality holds:
\begin{eqnarray}\label{wls1}
\mathcal{E}^{\varepsilon}_B\left(\Qvec, \Mvec\right)\le \liminf_{k\rightarrow\infty} \mathcal{E}^{\varepsilon}_B\left(\Qvec_k, \Mvec_k\right),
\end{eqnarray}whenever $\left(\Qvec_k, \Mvec_k\right)\rightharpoonup \left(\Qvec, \Mvec\right) \mbox{ weakly in } W^{1,2}\left(\Omega; \mathbb{R}^{3\times 3}\right)\times W^{1,2}\left(\Omega; \mathbb R^3\right) \mbox{ as } k\to\infty.$
\end{proposition}
\begin{proof}
Set 
\begin{eqnarray}
\mathscr{G}\left(\Qvec,\Mvec\right)
=f_B\left(\Qvec, \Mvec\right) + \frac{\mu_0}{2}|\Hvec_{\Mvec}|^2 - \mu_0\Mvec\cdot \Hvec_{ext}
- \frac{\mu_0\gamma}{2}\Qvec\Mvec\cdot\Mvec\nonumber\\ - \frac{\chi_1\mu_0}{2}\Qvec\Hvec_{\Mvec}\cdot\Hvec_{\Mvec} - \frac{\chi_1\mu_0}{2}\Qvec\Hvec_{ext}\cdot\Hvec_{ext}.\nonumber
\end{eqnarray} 
We assume that $\mathscr{G}\left(\Qvec,\Mvec\right)$ is non-negative by the estimate \eqref{lb:5}.
Thanks to the lower semicontinuity property of $L^2$-norm, we have the following liminf inequality:
\begin{eqnarray}
||\nabla \Qvec||^2_{L^2\left(\Omega\right)} \le  \liminf_{k\rightarrow\infty}||\nabla \Qvec_k||^2_{L^2\left(\Omega\right)},\label{eq:ex9}\\
||\nabla \Mvec||^2_{L^2\left(\Omega\right)}\le  \liminf_{k\rightarrow\infty}||\nabla \Mvec_k||^2_{L^2\left(\Omega\right)}. \label{eq:ex10}
\end{eqnarray}
Thanks to the Rellich-Kondra\v{s}ov's compact embedding (see e.g., \cite[Theorem 6.1-5]{ciarlet2021mathematical}), it holds   
\begin{alignat}{5}
&\Qvec_k& &\rightarrow\Qvec &~\mbox{ strongly in }& L^q\left(\Omega;\mathbb R^{3\times 3}\right) &\quad q\in [1, \infty)&,\\
&\Mvec_k& &\rightarrow\Mvec &~\mbox{ strongly in }& L^q\left(\Omega;\mathbb R^3\right) &\quad q\in [1, \infty)&.
\end{alignat}\label{eq:ex12}
Up to a subsequence (without reindexing), it follows that
\begin{alignat}{5}
&\Qvec_k& &\rightarrow \Qvec& \mbox{ a.e. in } \Omega,\label{eq:ex13}\\
&\Mvec_k& &\rightarrow \Mvec& \mbox{ a.e. in } \Omega.\label{eq:ex14}
\end{alignat}
Using the closure property of pointwise convergences \eqref{eq:ex13} and \eqref{eq:ex14} and then applying Fatou's lemma (e.g.\ \cite{Rindler}) to $\mathscr{G}\left(\Qvec,\Mvec\right)$, it holds that
\begin{eqnarray}
\int_{\Omega}\mathscr{G}\left(\Qvec,\Mvec\right)dx\le\liminf_{k\rightarrow\infty}\int_{\Omega}\mathscr{G}\left(\Qvec_k,\Mvec_k\right)\d x.\label{eq:ex16}
\end{eqnarray} 
Moreover, thanks to Proposition \ref{lem-str}, it holds that
\begin{eqnarray}
\Hvec_{\Mvec_k}\rightharpoonup\Hvec_{\Mvec} \mbox{ weakly in } L^2\left(\mathbb R^3; \mathbb R^3\right).
\end{eqnarray}
Thanks to the lower-semicontinuity property of $L^2$-norm, we have
\begin{eqnarray}
||\Hvec_{\Mvec}||^2_{L^2\left(\mathbb R^3;~\mathbb R^3\right)}\le  \liminf_{k\rightarrow\infty}||\Hvec_{\Mvec_k}||^2_{L^2\left(\mathbb R^3;~\mathbb R^3\right)}. \label{eq:ex17}
\end{eqnarray}
Assembling \eqref{eq:ex17} and \eqref{eq:ex16}, we infer the claim \eqref{wls1}.
\end{proof}
To this end, we state and prove the existence of energy minimizers for the penalized energy functional (cf. \eqref{penalized_energy}).
\begin{theorem}\label{existence_in_limit}
Let $\Omega\subset\mathbb R^3$ be a bounded domain such that $\partial\Omega$ is Lipschitz and $\mathcal{Q}\times\mathcal{M}\ne\emptyset$. Also, we consider that $\Hvec_{ext}\in C\left(\overline{\Omega}; \mathbb R^3\right)$. Then the penalized energy functional $\mathcal{E}^{\varepsilon}_B\left(\Qvec, \Mvec\right)$ (cf. \eqref{penalized_energy}) has a minimum on the space $\mathcal{Q}_B\times\mathcal{M}_B$.  
\end{theorem}
\begin{proof}
Applying Proposition \ref{prop:comp} and Proposition \ref{prop:wls}, we conclude \Cref{existence_in_limit} due to the direct method in the calculus of variations (see e.g.\ \cite{Rindler}). 
\end{proof}
In the following, we show that the sequence of minimizers of the penalized energy functional \eqref{penalized_energy} indeed converges to a minimizer of the actual ferronematic energy \eqref{bulk_energy}.
\begin{proposition}\label{energy_bound}
Let $\Omega\subset\mathbb R^3$ be a bounded domain such that $\partial\Omega$ is Lipschitz. Assume that $\left(\Qvec^{\varepsilon}, \Mvec^{\varepsilon}\right)$ be a sequence of minimizer of the functional $\mathcal{E}^{\varepsilon}_B$. Then, up to subsequences 
\begin{eqnarray}
\Qvec^{\varepsilon}\rightharpoonup\Tilde{\Qvec} &\mbox{ weakly in }& W^{1,2}\left(\Omega; \mathbb R^{3\times 3}\right),\\
\Mvec^{\varepsilon}\rightharpoonup\Tilde{\Mvec} &\mbox{ weakly in }& W^{1,2}\left(\Omega; \mathbb R^3\right).
\end{eqnarray}
Moreover, it can be shown that $\left(\Tilde{\Qvec}, \Tilde{\Mvec}\right)$ is a minimizer of the energy functional $\mathcal{E}_B\left(\Qvec, \Mvec\right)$.
\end{proposition}
\begin{proof}
Assume by \Cref{existence_in_limit} that $\left(\Qvec^{\varepsilon}, \Mvec^{\varepsilon}\right)$ is a sequence of minimizers of the penalised energy functional $\mathcal{E}^{\varepsilon}_B\left(\Qvec, \Mvec\right)$. By the property of the minimum energy level, we have
\begin{eqnarray}
\mathcal{E}^{\varepsilon}_B\left(\Qvec^{\varepsilon}, \Mvec^{\varepsilon}\right)\le \mathcal{E}^{\varepsilon}_B\left(\Qvec, \Mvec\right)<\infty
\end{eqnarray}
for any $\left(\Qvec, \Mvec\right)\in\mathcal{Q}_B\times\mathcal{M}_B$. Using \eqref{lb:1}, \eqref{lb:2}, \eqref{lb:3}, \eqref{lb:4} and \eqref{lb:5} in \eqref{penalized_energy}, we deduce that
\begin{eqnarray}
||\Qvec^{\varepsilon}||_{W^{1,2}\left(\Omega; \mathbb R^{3\times 3}\right)}\le C_1 \mbox{ and } ||\Mvec^{\varepsilon}||_{W^{1,2}\left(\Omega; \mathbb R^3\right)}\le C_2\label{eq:eb1}
\end{eqnarray}
for any $\varepsilon>0$.
It then follows from \eqref{eq:eb1}, up to a subsequence (without reindexing), that 
\begin{eqnarray}
\Qvec^{\varepsilon}\rightharpoonup\Tilde{\Qvec} &\mbox{ weakly in }& W^{1,2}\left(\Omega; \mathbb R^{3\times 3}\right),\label{eq:eb2}\\
\Mvec^{\varepsilon}\rightharpoonup\Tilde{\Mvec} &\mbox{ weakly in }& W^{1,2}\left(\Omega; \mathbb R^{3}\right)\label{eq:eb3}.
\end{eqnarray}
By Rellich-Kondra\v{s}ov compact embedding $W^{1,2}\hookrightarrow L^6$ (see e.g.\ \cite{ciarlet2021mathematical}), we have
\begin{eqnarray}
\Qvec^{\varepsilon}\rightarrow\Tilde{\Qvec} &\mbox{ strongly in }& L^{6}\left(\Omega; \mathbb R^{3\times 3}\right)\label{eq:eb4},\\
\Mvec^{\varepsilon}\rightarrow\Tilde{\Mvec} &\mbox{ strongly in }& L^{6}\left(\Omega; \mathbb R^{3}\right)\label{eq:eb5}.
\end{eqnarray}
As $\Mvec^{\varepsilon}\in W^{1,2}$ is a minimizer of $\mathcal{E}^{\varepsilon}_B$, the sequence of stray field $\Hvec_{\Mvec^{\varepsilon}}$ appears as a weak solution of \eqref{25} corresponding to $\Mvec^{\varepsilon}$. Applying the elliptic regularity (cf. \cite{gilbarg1998elliptic}), there exists $C_3>0$ such that the following estimate holds
\begin{eqnarray}
||\Hvec_{\Mvec^{\varepsilon}}||_{W^{1,2}\left(\mathbb R^3; \mathbb R^3\right)}\le C_3\left(||\nabla\cdot\Mvec^{\varepsilon}||_{L^2\left(\Omega\right)}+||\Phi^{\varepsilon}||_{L^2}\right)<\infty
\end{eqnarray}
for any $\varepsilon>0$. The last implication follows due to \eqref{eq:eb1} and the uniform bound in the weak solution $\Hvec_{\Mvec^{\varepsilon}}$.
Therefore, up to a subsequence
\begin{eqnarray}
\Hvec_{\Mvec^{\varepsilon}}\rightharpoonup\Hvec_{\Tilde{\Mvec}} \mbox{ weakly in } W^{1,2}\left(\mathbb R^3; \mathbb R^3\right).
\end{eqnarray}
Again, thanks to the Rellich-Kondra\v{s}ov compact embedding $W^{1,2}\hookrightarrow L^4$, 
we possess
\begin{eqnarray}
\Hvec_{\Mvec^{\varepsilon}}\rightarrow\Hvec_{\Tilde{\Mvec}} \mbox{ strongly in } L^4\left(\mathbb R^3; \mathbb R^3\right).\label{eq:str_L4}
\end{eqnarray}
From \eqref{eq:str_L4}, it follows that
\begin{eqnarray}
\varepsilon||\Hvec_{\Mvec^{\varepsilon}}||_{L^4\left(\mathbb R^3;~\mathbb R^3\right)}\rightarrow 0 \mbox{ as } \varepsilon\rightarrow 0.\label{Str-eq1}
\end{eqnarray}
To conclude the theorem, we need to obtain strong convergence in $\nabla\Qvec^{\varepsilon}$ and $\nabla\Mvec^{\varepsilon}$, while other energy terms can be treated via already obtained $L^6$-convergence in $\Qvec^{\varepsilon}$ and $\Mvec^{\varepsilon}$. It is therefore the task to show that
\begin{eqnarray}
\nabla\Qvec^{\varepsilon}\rightarrow\nabla\Tilde{\Qvec} \mbox{ strongly in } L^2\left(\Omega; \mathbb R^{3\times 3}\right),\label{str-Q}\\
\nabla\Mvec^{\varepsilon}\rightarrow\nabla\Tilde{\Mvec} \mbox{ strongly in } L^2\left(\Omega; \mathbb R^3\right)\label{str-M}.
\end{eqnarray}
To this aim, we compare $\mathcal{E}^{\varepsilon}\left(\Qvec^{\varepsilon}, \Mvec^{\varepsilon}\right)$ to $\mathcal{E}^{\varepsilon}\left(\Tilde{\Qvec}, \Tilde{\Mvec}\right)$, yields
\begin{eqnarray}
& &\int_{\Omega} \biggl[\frac{K_1}{2}|\nabla \Qvec^{\varepsilon}|^2 + \frac{K_2}{2}|\nabla \Mvec^{\varepsilon}|^2 + f_B\left(\Qvec^{\varepsilon}, \Mvec^{\varepsilon}\right) + \frac{\mu_0}{2}|\Hvec_{\Mvec^{\varepsilon}}|^2 - \mu_0\Mvec^{\varepsilon}\cdot \Hvec_{ext}\nonumber\\
& &~~~~~~~- \frac{\chi_1\mu_0}{2}\Qvec^{\varepsilon}\Hvec_{\Mvec^{\varepsilon}}\cdot\Hvec_{\Mvec^{\varepsilon}} - \frac{\chi_1\mu_0}{2}\Qvec^{\varepsilon}\Hvec_{ext}\cdot\Hvec_{ext}+\frac{\varepsilon\mu_0}{2}|\Hvec_{\Mvec^{\varepsilon}}|^4\biggr]\nonumber\\
&\le&\int_{\Omega} \biggl[\frac{K_1}{2}|\nabla \Tilde{\Qvec}|^2 + \frac{K_2}{2}|\nabla \Tilde{\Mvec}|^2 + f_B\left(\Tilde{\Qvec}, \Tilde{\Mvec}\right) + \frac{\mu_0}{2}|\Hvec_{\Tilde{\Mvec}}|^2 - \mu_0\Tilde{\Mvec}\cdot \Hvec_{ext}\nonumber\\
& &~~~~~~~- \frac{\chi_1\mu_0}{2}\Tilde{\Qvec}\Hvec_{\Tilde{\Mvec}}\cdot\Hvec_{\Tilde{\Mvec}} - \frac{\chi_1\mu_0}{2}\Tilde{\Qvec}\Hvec_{ext}\cdot\Hvec_{ext}+\frac{\varepsilon\mu_0}{2}|\Hvec_{\Tilde{\Mvec}}|^4\biggr].\;\label{en-comp1}
\end{eqnarray}
Using \eqref{eq:eb4}, \eqref{eq:eb5}, \eqref{Str-eq1} in \eqref{en-comp1}, we deduce that
\begin{eqnarray}
& &\int_{\Omega} \frac{K_1}{2}|\nabla \Qvec^{\varepsilon}|^2 + \frac{K_2}{2}|\nabla \Mvec^{\varepsilon}|^2\nonumber\\
&\le& \int_{\Omega} \frac{K_1}{2}|\nabla \Tilde{\Qvec}|^2 + \frac{K_2}{2}|\nabla \Tilde{\Mvec}|^2 + o\left(\varepsilon\right).\label{en-comp2}
\end{eqnarray}
From \eqref{eq:eb2} and \eqref{eq:eb3}, we write
\begin{eqnarray}
\nabla\Mvec^{\varepsilon}=\nabla\Tilde{\Mvec} + A^{\varepsilon} \mbox{ where } A^{\varepsilon}\rightharpoonup 0 \mbox{ weakly in } L^2\left(\Omega; \mathbb R^{3\times 3}\right),\label{eq:M1}\\
\nabla\Qvec^{\varepsilon}=\nabla\Tilde{\Qvec} + B^{\varepsilon} \mbox{ where } B^{\varepsilon}\rightharpoonup 0 \mbox{ weakly in } L^2\left(\Omega; \mathbb R^{3\times 3}\right)\label{eq:M2}.
\end{eqnarray}
Using \eqref{eq:M1} and \eqref{eq:M2} in \eqref{en-comp2}, it follows that
\begin{eqnarray}
& &\int_{\Omega} \frac{K_1}{2}\left(|\nabla\Tilde{\Qvec}|^2+2\nabla\Tilde{\Qvec}\cdot A^{\varepsilon} + |A^{\epsilon}|^2\right) + \frac{K_2}{2}\left(|\nabla\Tilde{\Mvec}|^2+2\nabla\Tilde{\Mvec}\cdot B^{\varepsilon} + |B^{\epsilon}|^2\right)\nonumber\\
&\le& \int_{\Omega} \frac{K_1}{2}|\nabla \Tilde{\Qvec}|^2 + \frac{K_2}{2}|\nabla \Tilde{\Mvec}|^2 + o\left(\varepsilon\right).\label{en_comp3}
\end{eqnarray}
Since $A^{\varepsilon}\rightharpoonup 0 \mbox{ weakly in } L^2\left(\Omega; \mathbb R^3\right)$ and $B^{\varepsilon}\rightharpoonup 0 \mbox{ weakly in } L^2\left(\Omega; \mathbb R^{3\times 3}\right)$, by the definition of weak convergence, it holds that
\begin{eqnarray}
\int_{\Omega}A^{\varepsilon}\cdot\nabla\Tilde{\Mvec}=0\mbox{ and }\int_{\Omega}B^{\varepsilon}\cdot\nabla\Tilde{\Qvec}=0 \mbox{ as } \varepsilon\rightarrow 0.
\end{eqnarray}
We then conclude from \eqref{en_comp3} that
\begin{eqnarray}
\int_{\Omega}|A^{\varepsilon}|^2 \rightarrow 0 \mbox{ and }\int_{\Omega}|B^{\varepsilon}|^2 \rightarrow 0 \mbox{ as } \varepsilon\rightarrow 0.\label{eq:norm_conv}
\end{eqnarray}
Weak convergences \eqref{eq:M1} and \eqref{eq:M2} and convergence in norm \eqref{eq:norm_conv}, together, imply
\begin{eqnarray}
A^{\varepsilon}\rightarrow 0 \mbox{ and }B^{\varepsilon}\rightarrow 0 \mbox{ strongly in } L^2.
\end{eqnarray}
Therefore, the strong convergences \eqref{str-Q} and \eqref{str-M} follow.
Hence we conclude that 
\begin{eqnarray}
\lim_{\varepsilon\rightarrow 0}\mathcal{E}^{\varepsilon}_B\left(\Qvec^{\varepsilon}, \Mvec^{\varepsilon}\right)\rightarrow \mathcal{E}_B\left(\Tilde{\Qvec}, \Tilde{\Mvec}\right).
\end{eqnarray}
By comparing $\mathcal{E}^{\varepsilon}_B\left(\Qvec^{\varepsilon}, \Mvec^{\varepsilon}\right)$ with $\mathcal{E}^{\varepsilon}_B\left(\Qvec, \Mvec\right)$ for any $\left(\Qvec, \Mvec\right)\in\mathcal{A}_B\times\mathcal{S}_B$, we assert that $\left(\Tilde{\Qvec}, \Tilde{\Mvec}\right)$ is a minimizer of $\mathcal{E}_B\left(\Qvec, \Mvec\right)$.
\end{proof}
In what follows, we discuss the complete derivation of a two-dimensional ferronematic model obtained as a dimension-reduced version of the three-dimensional ferronematic energy (cf. \eqref{bulk_energy}). The technique is based on the weak convergence method in the dimension reduction (see e.g.\ \cite{braides2002gamma}). The two main aspects are the following: to unify the two regimes in the limit $\eta\rightarrow 0$, and to achieve sufficient regularities in the reduction of the magnetostatic term, so that it allows us to pass to the limit in the nematic-magnetostatic coupling term.
\section{Derivation of the limiting ferronematic energy}\label{derivation of the limiting ferronematic energy}
In this part, we present the derivation of a limiting ferronematic energy in a thin film setting. As already mentioned, the obtained reduced energy is already discussed in \cite{dutta2025study, dutta2026existence}, and we here present a complete explanation. To this aim, we let $\omega\subset\mathbb R^2$ be open and bounded. We assume a ferronematic thin film (cf. \Cref{fig:geo}), denoted by
$$\Omega_{\eta}\equiv\{\left(x_1, x_2\right)\in\omega, 0<x_3<\eta\},$$ where $\eta$ is a constant that measures the thickness of the thin film $\Omega_{\eta}$ and $\eta<<1$. We then rescale the domain $\Omega_{\eta}$ to $\Omega$. Here we assume that $\Omega$ is a cylinder whose height is of unit length, and $\omega$ is the cross-section as expressed below:
$$\Omega\equiv\{\left(y_1, y_2\right)\in\omega, 0<y_3<1\}.$$ We make a correspondance between $\Omega_{\eta}$ and $\Omega$ via an injective mapping $y: \Omega_{\eta}\rightarrow\Omega$ (cf. \Cref{fig:geo}) by
\begin{eqnarray}
y_1=x_1,~~y_2=x_2,~~y_3=\frac{1}{\eta}x_3.
\end{eqnarray}
The mapping $y$ corresponds each $\overline{\Mvec}:\Omega_{\eta}\rightarrow\mathbb R^3$ to $\Mvec:\Omega\rightarrow\mathbb R^3$ by $
\Mvec\left(y_1, y_2, y_3\right)=\overline{\Mvec}\left(x_1, x_2, x_3\right)$ forall $\left(x_1, x_2, x_3\right)\in\Omega_{\eta}$ and each $\overline{\Qvec}:\Omega_{\eta}\rightarrow\mathbb R^{3\times 3}$ to $\Qvec:\Omega\rightarrow\mathbb R^{3\times 3}$ by $\Qvec\left(y_1, y_2, y_3\right)=\overline{\Qvec}\left(x_1, x_2, x_3\right)$ forall $\left(x_1, x_2, x_3\right)\in\Omega_{\eta}$. In the same spirit, the corresponding magnetostatic scalar potential $\overline{\Phi}:\Omega_{\eta}\rightarrow\mathbb R$ is related to $\Phi:\Omega\rightarrow\mathbb R$ via
$\Phi\left(y_1, y_2, y_3\right)=\overline{\Phi}\left(x_1, x_2, x_3\right)$.
\begin{figure}
    \centering
    \includegraphics[width=0.45\linewidth]{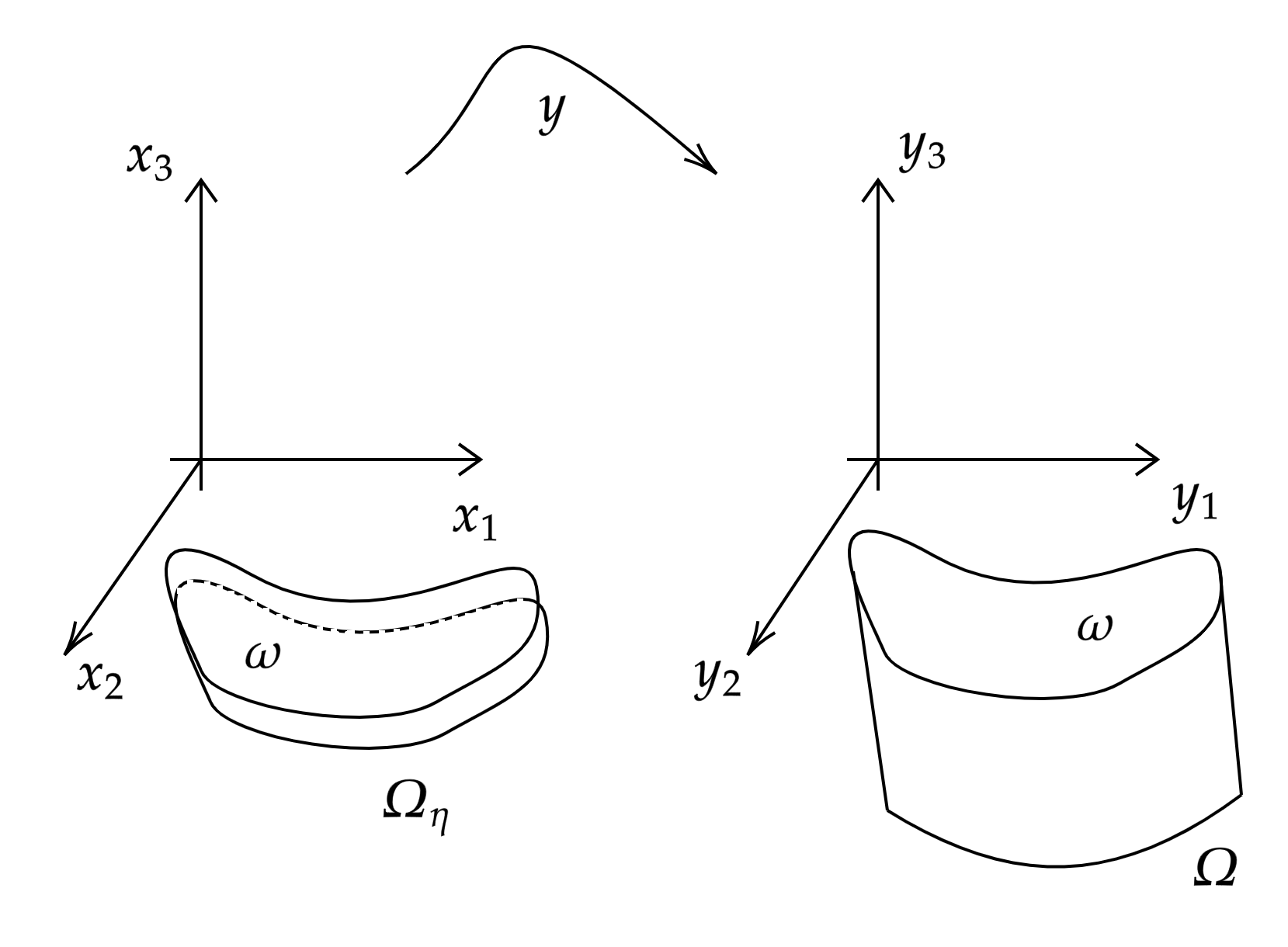}
    \caption{Physical sketch of a thin film $\Omega_{\eta}$ under the mapping $y$}
    \label{fig:geo}
\end{figure}
In addition, we assume the contribution due to the surface energy in this thin film regime (see, e.g.\ \cite{golovaty2015dimension}) as
\begin{eqnarray}
f_{sur}\left(\Qvec, \nu\right)=\beta_1\left(\Qvec\nu\cdot\nu\right)+\beta_2\Qvec\cdot\Qvec+\beta_3\left(\Qvec\nu\cdot\nu\right)^2+\beta_4|\Qvec\nu|^2,
\end{eqnarray}
where $\beta_1, \beta_2, \beta_3$ and $\beta_4$ are constants, and $\nu\in\mathbb S^2$ is the surface normal to the liquid crystal. Following \cite{golovaty2015dimension}, we consider a special regime of these constants and write the surface energy in the following form
\begin{eqnarray}
f_{sur}\left(\Qvec, \hat{x}_3\right)=\alpha[\left(\Qvec\hat{x}_3\cdot\hat{x}_3\right)-\beta]^2+\gamma|\left(\mathbb I - \hat{x}_3\otimes\hat{x}_3\right)\Qvec\hat{x}_3|^2.
\end{eqnarray}
In the same spirit of \cite{golovaty2015dimension}, we consider the asymptotic regime of the surface energy as
\begin{eqnarray}
f_{sur}\left(\Qvec, \hat{x}_3\right)=f^0_{sur}\left(\Qvec, \hat{x}_3\right)+\eta f^1_{sur}\left(\Qvec, \hat{x}_3\right),
\end{eqnarray}
where
\begin{eqnarray}
f^0_{sur}\left(\Qvec, \hat{x}_3\right)=\alpha'[\left(\Qvec\hat{x}_3\cdot\hat{x}_3\right)-\beta]^2+\gamma'|\left(\mathbb I - \hat{x}_3\otimes\hat{x}_3\right)\Qvec\hat{x}_3|^2,\newline\\
f^1_{sur}\left(\Qvec, \hat{x}_3\right)=\alpha''[\left(\Qvec\hat{x}_3\cdot\hat{x}_3\right)-\beta]^2+\gamma''|\left(\mathbb I - \hat{x}_3\otimes\hat{x}_3\right)\Qvec\hat{x}_3|^2.
\end{eqnarray}
Also, we assume that $\gamma$ is independent of $\eta$, and for this purpose $\gamma''=0$. It can then be observed that the surface energy $f_{sur}\left(\Qvec, \hat{x}_3\right)$ will be minimized when the nematic order parameter $\Qvec$ satisfies $\Qvec\hat{x}_3=\beta\hat{x}_3$. It is not difficult to observe that $\Qvec$ will then be independent of the third direction.
\begin{remark}
In the above, we presume that the surface energy term is minimized in a regime where the thin film limit has a special nematic structure, i.e., the nematic order parameter $\Qvec$ is described by the two independent variables. We refer to \cite{golovaty2014minimizers} for this description. Our focus here is to ensure the limit passages in the dimension reduction process due to the involvement of the stray field or magnetostatic energy, $\int_{\mathbb R^3}\frac{\mu_0}{2}|\Hvec_{\Mvec}|^2$ and the nematic-magnetostatic coupling energy $-\int_{\mathbb R^3}\frac{\chi_1\mu_0}{2}\Qvec\Hvec_{\Mvec}\cdot\Hvec_{\Mvec}$.
\end{remark}
Using the correspondences due to the transformation $y$, the ferronematic energy $\mathcal{\overline{E}}_{\eta}\left(\overline{\Qvec}, \overline{\Mvec}\right)$ over $\Omega_{\eta}$ can be related to the ferronematic energy $\mathcal{E}_{\eta}\left(\Qvec, \Mvec\right)$ over $\Omega$ as follows
\begin{eqnarray}
& &\mathcal{\overline{E}}_{\eta}\left(\overline{\Qvec}, \overline{\Mvec}\right)\nonumber\\
&=&\int_{\Omega_{\eta}}\biggl[\frac{K_1}{2}|\nabla\overline{\Qvec}|^2+\frac{K_2}{2}|\nabla\overline{\Mvec}|^2+f_B\left(\overline{\Qvec}, \overline{\Mvec}\right)+\frac{\mu_0}{2}|\nabla\overline{\Phi}|^2- \frac{\chi_1\mu_0}{2}\overline{\Qvec}\nabla\overline{\Phi}\cdot\nabla\overline{\Phi}\nonumber\\
& &\nonumber\\
& &~~~-\mu_0\overline{\Mvec}\cdot\overline{\Hvec_{ext}} - \frac{\chi_1\mu_0}{2}\overline{\Qvec}\overline{\Hvec_{ext}}\cdot\overline{\Hvec_{ext}}+f_{sur}\left(\overline{\Qvec}, \hat{x}_3\right)\biggr]\nonumber\\
&=&\int_{\Omega}\frac{K_1}{2}\biggl[\left(|\nabla_p\Qvec|^2+\frac{1}{\eta^2}|\Qvec_{,3}|^2\right)+\frac{K_2}{2}\left(|\nabla_p\Mvec|^2+\frac{1}{\eta^2}|\Mvec_{,3}|^2\right)+f_B\left(\Qvec, \Mvec\right)\nonumber\\
& &+\frac{\mu_0}{2}\left(|\nabla_p\Phi|^2+\frac{1}{\eta^2}|\Phi_{,3}|^2\right)- \frac{\chi_1\mu_0}{2}\Qvec\left(\nabla_p\Phi+\frac{1}{\eta}\Phi_{,3}e_3\right)\cdot\left(\nabla_p\Phi+\frac{1}{\eta}\Phi_{,3}e_3\right)\nonumber\\
& & -\mu_0\Mvec\cdot\Hvec_{ext} - \frac{\chi_1\mu_0}{2}\Qvec\Hvec_{ext}\cdot\Hvec_{ext}+\frac{1}{\eta}f_{sur}\left(\Qvec, \hat{x}_3\right)\biggr]\nonumber\\
&:=& \mathcal{E}_{\eta}\left(\Qvec, \Mvec\right),\label{eq:energy_reduced}
\end{eqnarray}
which is subjected to the transformed stationary Maxwell's equation
\begin{eqnarray}
\nabla_p\cdot\left(-\nabla_p\Phi+\Mvec_p\right)+\frac{1}{\eta}\left(-\frac{1}{\eta}\Phi_{,3}+M_3\right)_{,3}=0 \mbox{ on } \mathbb R^3.\label{eq:reduced_maxwell}
\end{eqnarray}
In the same spirit, we correspond the transformed stationary Maxwell's equation (cf. \eqref{eq:reduced_maxwell}) to the following variational problem
\begin{eqnarray}
\min_{\Phi\in \mathcal{V}}\frac{\mu_0}{2}\int_{\mathbb R^3}\bigg|\nabla_p\Phi-\Mvec_p\bigg|^2 + \bigg|\frac{1}{\eta}\Phi_{,3}-M_3\bigg|^2,\label{stray_minimization}
\end{eqnarray}
where
\begin{eqnarray}
\mathcal{V}=\bigg\{\phi:\mathbb R^3\rightarrow\mathbb R, \nabla\phi\in L^2\left(\mathbb R^3; \mathbb R^3\right), \int_{\mathcal{B}}\phi=0\bigg\}
\end{eqnarray}
and $p=1,2$.
In this setting, we define the admissible space as follows:
\begin{eqnarray}
\mathcal{Q}_{\eta}=\bigg\{\Qvec\in W^{1,2}\left(\Omega_{\eta}; \mathcal{Q}\right):\Qvec|_{\partial\omega\times z}= g_1~~\forall z\in\left(0, \eta\right)\bigg\}\nonumber\newline\\
\mathcal{M}_{\eta}=\bigg\{\Qvec\in W^{1,2}\left(\Omega_{\eta}; \mathbb R^3\right):\Mvec|_{\partial\omega\times z}= g_2~~\forall z\in\left(0, \eta\right)\bigg\}\nonumber\newline\\
\mathcal{Q}_{\eta}\times\mathcal{M}_{\eta}=\bigg\{\left(\Qvec, \Mvec\right): \Qvec\in\mathcal{Q}_{\eta}, \Mvec\in\mathcal{M}_{\eta}\bigg\},
\end{eqnarray}
where $g_1, g_2$ are prescribed such a way that $\mathcal{Q}_{\eta}\times\mathcal{M}_{\eta}$ is nonempty.
\begin{remark}
We note that the assumption that the admissible space is nonempty is essential in this case, while the admissible space defined on a bulk setting is nonempty by an example \cite{bethuel1994ginzburg}.
\end{remark}
In the following, we explain the dimension reduction of the bulk energy; we first tackle the magnetostatic part of the energy and then present the reduction of state variables. 
\subsection{Limiting magnetostatic energy}
To derive the limiting ferronematic energy in the thin-film setting from the non-local ferronematic energy in bulk (cf. \eqref{bulk_energy}), we first obtain the limiting magnetostatic energy, following \cite{Giogia}. We then obtain the improved convergence with the higher regularity, which allows us to pass to the limit in the nematic-stray field coupling term.\newline\\
To this aim, we first state the following existence theorem. Given the space $\mathcal{A}_{\eta}\times\mathcal{S}_{\eta}$ is nonempty, the proof of the theorem follows straightforwardly.
\begin{theorem}\label{th:minimizer}
Let $\Omega_{\eta}\subset\mathbb R^3$ be a bounded domain such that $\mathcal{Q}_{\eta}\times\mathcal{M}_{\eta}\ne\emptyset$. Moreover, we assume that $\Hvec_{ext}\in C\left(\overline{\Omega}_{\eta}; \mathbb R^3\right)$. Then the functional $\mathcal{E}_{\eta}\left(\Qvec, \Mvec\right)$ (cf. \eqref{bulk_energy}) has a minimum on the space $\mathcal{Q}_{\eta}\times\mathcal{M}_{\eta}$.
\end{theorem}
\begin{proof}
The proof follows by the direct method in the calculus of variations.
\end{proof}
In what follows, we derive the limiting magnetostatic energy exactly following Gioia and James approximation \cite{Giogia} in the regime of minimum surface energy of NLCs. To tackle also the nematic-magnetostatic coupling term, we obtain the better convergence of the scalar potential in Proposition \ref{improved_convergence}.
\begin{proposition}\label{prop:reduced_stray}
Let $\left(\Qvec_{\eta}, \Mvec_{\eta}\right)$ be a sequence of minimizers of the energy functional $\mathcal{E}_{\eta}\left(\Qvec, \Mvec\right)$. Also, let $\Mvec_{\eta}\rightarrow\check{\Mvec}$ in $L^2\left(\Omega; \mathbb R^3\right)$ and $\Phi_{\eta}$ be the corresponding magnetostatic potential, i.e., the solution of \eqref{eq:reduced_maxwell} related to $\Mvec_{\eta}$. Then the following convergences hold
\begin{eqnarray}
\nabla_p\Phi_{\eta}&\rightarrow& 0 \mbox{ strongly in } L^2\left(\mathbb R^3\right),\label{eq:reduced_stray1}\\
\frac{1}{\eta}{\Phi_{\eta}}_{,3}&\rightarrow& \check{M_3} \mbox{ strongly in } L^2\left(\mathbb R^3\right),\label{eq:reduced_stray2}\\
\frac{\mu_0}{2}\int_{\Omega}\left(|\nabla_p\Phi_{\eta}|^2+\frac{1}{\eta^2}|{\Phi_{\eta}}_{,3}|^2\right)&\rightarrow&\frac{\mu_0}{2}\int_{\Omega}\check{M_3}^2.\label{eq:str2'}
\end{eqnarray}
\end{proposition}
\begin{proof}
Assume by \Cref{th:minimizer} that $\left(\Qvec_{\eta}, \Mvec_{\eta}\right)$ be a sequence of minimizers of the energy functional $\mathcal{E}_{\eta}\left(\Qvec, \Mvec\right)$. It therefore follows that 
\begin{eqnarray}
\mathcal{E}_{\eta}\left(\Qvec_{\eta}, \Mvec_{\eta}\right)\le\mathcal{E}_{\eta}\left(\Qvec, \Mvec\right)<\infty\label{eq:str1}
\end{eqnarray}
for any fixed $\left(\Qvec, \Mvec\right)\in \mathcal{A}_{\eta}\times \mathcal{S}_{\eta}$ such that $f^0_{sur}\left(\Qvec, \hat{x}_3\right)=0$ or $\Qvec\hat{x}_3=\beta\hat{x}_3$ and $\Mvec_{,3}=0$. Therefore, up to a subsequence (without re-indexing)
\begin{eqnarray}
\Qvec_{\eta}&\rightharpoonup&\check{\Qvec}\mbox{ weakly in } W^{1,2}\left(\Omega; \mathbb R^{3\times 3}\right),\label{eq:Q}\\
\Mvec_{\eta}&\rightharpoonup&\check{\Mvec}\mbox{ weakly in } W^{1,2}\left(\Omega; \mathbb R^{3}\right).\label{eq:M}
\end{eqnarray}
Thanks again to the Kondra\v{s}ov compact embedding $W^{1,2}\hookrightarrow L^2$ (e.g.\ \cite{ciarlet2021mathematical}), up to a subsequence (without re-indexing), we have
\begin{eqnarray}
\Qvec_{\eta}&\rightarrow&\check{\Qvec}\mbox{ strongly in } L^{2}\left(\Omega; \mathbb R^{3\times 3}\right),\label{eq:Q'}\\
\Mvec_{\eta}&\rightarrow&\check{\Mvec}\mbox{ strongly in } L^{2}\left(\Omega; \mathbb R^{3}\right).\label{eq:M'}
\end{eqnarray}
Since $\left(\Qvec_{\eta}, \Mvec_{\eta}\right)$ be a sequence of minimizers of the energy functional $\mathcal{E}_{\eta}\left(\Qvec, \Mvec\right)$, the corresponding sequence of scalar potential $\Phi_{\eta}$ are the minimizers of the variational problem \eqref{stray_minimization} on the space $V$. Therefore, a comparison of minimum potential related to $\Phi_{\eta}$ with the potential for $\Phi=0$, yields
\begin{eqnarray}
\frac{\mu_0}{2}\int_{\mathbb R^3}\bigg|\nabla_p\Phi_{\eta}-\left(\Mvec_p\right)_{\eta}\bigg|^2 + \bigg|\frac{1}{\eta}\Phi_{\eta,3}-\left(M_{3}\right)_{\eta}\bigg|^2\le C|\Omega|\label{eq:str2}
\end{eqnarray}
for some generic constant $C>0$. The estimate \eqref{eq:str2} follows by the fact that $||\Mvec_{\eta}||_{L^2}<\infty$ in $\eta>0$ due to the uniform bound of the energy functional $\mathcal{E}_{\eta}\left(\Qvec, \Mvec\right)<\infty$ in $\eta>0$.
As a consequence, we have the estimates
\begin{eqnarray}
||\nabla_p\Phi_{\eta}||_{L^2\left(\mathbb R^3\right)}\le C_1, \bigg|\bigg|\frac{1}{\eta}\Phi_{\eta,3}\bigg|\bigg|_{L^2\left(\mathbb R^3\right)}\le C_2\label{eq:str4}
\end{eqnarray}
in any $\eta>0$. Therefore, up to a subsequence (without re-indexing),
\begin{eqnarray}
\nabla_p\Phi_{\eta}&\rightharpoonup&\nabla_p\Phi \mbox{ weakly in } L^2\left(\mathbb R^3\right),\label{eq:str5}\\
\frac{1}{\eta}\Phi_{\eta,3}&\rightharpoonup& \zeta \mbox{ weakly in } L^2\left(\mathbb R^3\right).\label{eq:str6}
\end{eqnarray}
From \eqref{eq:str5} and \eqref{eq:str6}, we can write
\begin{eqnarray}
& &\nabla_p\Phi_{\eta}=\nabla_p\Phi+\left(a_p\right)_{\eta} \mbox{ such that } \left(a_p\right)_{\eta}\rightharpoonup 0 \mbox{ weakly in } L^2\label{eq:con_ap}\\
& &\frac{1}{\eta}\Phi_{\eta,3}=\zeta+a_{\eta}~~~~~~~~\mbox{ such that }~~~~ a_{\eta}\rightharpoonup 0~\mbox{ weakly in } L^2\label{eq:con_an}.
\end{eqnarray}
Moreover, \eqref{eq:str5} and \eqref{eq:str6} together, provides
\begin{eqnarray}
\nabla\Phi_{\eta}\rightharpoonup\nabla\Phi \mbox{ weakly in } L^2\left(\mathbb R^3\right)
\end{eqnarray}
along with $\Phi_{,3}=0 \mbox{ a.e. in } \mathbb R^3$. Applying Fubini's theorem, we can conclude that $\nabla\Phi_p=0$ a.e. in $\mathbb R^2$.
We will now show that
\begin{eqnarray}
\left(a_p\right)_{\eta}\rightarrow 0 \mbox{ and } a_{\eta}\rightarrow 0 \mbox{ strongly in } L^2\left(\mathbb R^3\right).\label{eq:a_p}
\end{eqnarray}
Since $L^2\left(\mathbb R^3\right)=\overline{C^{\infty}_{0}\left(\mathbb R^3\right)}$, we assume a sequence of test functions such that
\begin{eqnarray}
\check{\Mvec}^{\epsilon}\rightarrow\check{\Mvec} \mbox{ strongly in } L^2\left(\mathbb R^3\right) \mbox{ as } \epsilon\rightarrow 0.\label{eq:test_func}
\end{eqnarray}
To complete the proof, we construct the test functions of the scalar potential of the corresponding ${\check{M_3}}^{\epsilon}$. To this purpose, we can write
\begin{eqnarray}
\Phi^{\epsilon,\kappa}=\eta\int_{0}^{y_3}{\check{M_3}}^{\epsilon}\left(y_1, y_2, s\right)ds-\frac{\eta}{\kappa}\int_{0}^{y_3}\chi_{[1, 1+\kappa]}\left(s_1\right)ds_1\int_{0}^{1}{\check{M_3}}^{\epsilon}\left(y_1, y_2, s\right)ds + d^{\epsilon}.\label{eq:test_seq}
\end{eqnarray}
In the above, 
as before $\Phi_{\eta}$ is a sequence of solutions to the variational problem \eqref{stray_minimization}, yields
\begin{eqnarray}
& &\int_{\mathbb R^3}\bigg|\nabla_p\Phi_{\eta}-\left(\Mvec_p\right)_{\eta}\bigg|^2 + \bigg|\frac{1}{\eta}\Phi_{\eta,3}-\left(M_3\right)_{\eta}\bigg|^2\nonumber\newline\\
&\le&\int_{\mathbb R^3}\bigg|\nabla_p\Phi^{\epsilon,\kappa}_{\eta}-\left(\Mvec_p\right)_{\eta}\bigg|^2 + \bigg|\frac{1}{\eta}\Phi^{\epsilon,\kappa}_{\eta,3}-\left(M_3\right)_{\eta}\bigg|^2.\label{eq:inq_str}
\end{eqnarray}
Using \eqref{eq:test_seq} in the R.H.S of the above inequality (cf. \eqref{eq:inq_str}) and expanding the L.H.S, we obtain
\begin{eqnarray}
& &\int_{\mathbb R^3}\biggl[\left(\left(a_p\right)_{\eta}\right)^2-2\left(a_p\right)_{\eta}\cdot\left(\Mvec_p\right)_{\eta}+\left(\left(\Mvec_p\right)_{\eta}\right)^2+\left(\zeta-\left(M_3\right)_{\eta}\right)^2-2\left(\zeta-\left(M_3\right)_{\eta}\right)\cdot a_{\eta} + \left(a_{\eta}\right)^2\biggr]\nonumber\\
&\le&\int_{\mathbb R^3}\biggl[\bigg|\eta\nabla_p\int_{0}^{y_3}{\check{M_3}}^{\epsilon}\left(y_1, y_2, s\right)ds-\frac{\eta}{\kappa}\int_{0}^{y_3}\chi_{[1, 1+\kappa]}\left(s_1\right)ds_1\nabla_p\int_{0}^{1}{\check{M_3}}^{\epsilon}\left(y_1, y_2, s\right)ds-\left(\Mvec_p\right)_{\eta}\bigg|^2\nonumber\\
& &~~~~~~~~~~~~+\bigg|{\check{M_3}}^{\epsilon}\left(y_1, y_2, s\right)-\frac{1}{\kappa}\chi_{[1, 1+\kappa]}\left(s_1\right)ds_1\int_{0}^{1}{\check{M_3}}^{\epsilon}\left(y_1, y_2, s\right)ds-\left(M_3\right)_{\eta}\bigg|^2\biggr].\label{eq:inq_first}
\end{eqnarray}
We observe that the second term in the L.H.S of \eqref{eq:inq_str} can be treated as follows
\begin{eqnarray}
& &\int_{\mathbb R^3}\left(a_p\right)_{\eta}\cdot\left(\left(\Mvec_p\right)_{\eta}-\check{\Mvec_p}\right) + \left(a_p\right)_{\eta}\cdot\check{\Mvec_p}\nonumber\\
&\le&||\left(a_p\right)_{\eta}||_{L^2}||\left(\Mvec_p\right)_{\eta}-\check{\Mvec_p}||_{L^2}\rightarrow 0
\end{eqnarray}
as $\eta\rightarrow 0$. The last implication is asserted by $||\left(a_p\right)_{\eta}||_{L^2}<\infty$ in $\eta$ due to \eqref{eq:con_ap} and the convergence \eqref{eq:M'}.
Similarly the fifth term in the L.H.S  \eqref{eq:inq_str} can be dealt as
\begin{eqnarray}
& &\int_{\mathbb R^3}\left(\zeta-\left(M_3\right)_{\eta}\right)\cdot a_{\eta}\nonumber\\
&\le&\int_{\mathbb R^3}\zeta\cdot a_{\eta} + ||\left(M_3\right)_{\eta} - \check{M_3}||_{L^2}||a_{\eta}||_{L^2} - \int_{\mathbb R^3}\check{M_3}\cdot a_{\eta}\rightarrow 0
\end{eqnarray}
as $\eta\rightarrow 0$ and it follows by $\eqref{eq:con_an}$ and \eqref{eq:M'}. For any fixed $\epsilon$ and $\kappa$, the term in R.H.S of \eqref{eq:inq_first}
\begin{eqnarray}
\eta\nabla_p\int_{0}^{y_3}{\check{M_3}}^{\epsilon}\left(y_1, y_2, s\right)ds-\frac{\eta}{\kappa}\int_{0}^{y_3}\chi_{[1, 1+\kappa]}\left(s_1\right)ds_1\nabla_p\int_{0}^{1}{\check{M_3}}^{\epsilon}\left(y_1, y_2, s\right)ds\rightarrow 0
\end{eqnarray}
as $\eta\rightarrow 0$. Hence, the third term in L.H.S cancel out the first term in R.H.S. Therefore we finally have
\begin{eqnarray}
& &\limsup_{\eta\rightarrow 0}\int_{\mathbb R^3}\left(\left(a_p\right)_{\eta}\right)^2+\left(\zeta-\left(M_3\right)_{\eta}\right)^2 + \left(a_{\eta}\right)^2\nonumber\\
&\le&\int_{\mathbb R^3}\bigg|{\check{M_3}}^{\epsilon}\left(y_1, y_2, s\right)-\frac{1}{\kappa}\chi_{[1, 1+\kappa]}\left(s_1\right)ds_1\int_{0}^{1}{\check{M_3}}^{\epsilon}\left(y_1, y_2, s\right)ds-\check{M_3}\bigg|^2.\label{eq:f_inq}
\end{eqnarray}
An application of the triangle inequality to R.H.S of \eqref{eq:f_inq}, yields
\begin{eqnarray}
& &\limsup_{\eta\rightarrow 0}\int_{\mathbb R^3}\left(\left(a_p\right)_{\eta}\right)^2+\left(\zeta-\left(M_3\right)_{\eta}\right)^2 + \left(a_{\eta}\right)^2\nonumber\\
&\le&\int_{\mathbb R^3}\bigg|{\check{M_3}}^{\epsilon}-\check{M_3}\bigg|^2+\int_{\mathbb R^3}\bigg|\frac{1}{\kappa}\chi_{[1, 1+\kappa]}\left(s_1\right)ds_1\int_{0}^{1}{\check{M_3}}^{\epsilon}\left(y_1, y_2, s\right)ds\bigg|^2.\label{eq:f_inq2}
\end{eqnarray}
Here we note that $\int_{-\infty}^{+\infty}\chi^2_{\kappa, 1+\kappa}=\kappa$ and hence \eqref{eq:f_inq2} follows the inequality below:
\begin{eqnarray}
& &\limsup_{\eta\rightarrow 0}\int_{\mathbb R^3}\left(\left(a_p\right)_{\eta}\right)^2+\left(\zeta-\left(M_3\right)_{\eta}\right)^2 + \left(a_{\eta}\right)^2\nonumber\\
&\le&\int_{\mathbb R^3}\bigg|{\check{M_3}}^{\epsilon}-\check{M_3}\bigg|^2+\frac{1}{\kappa}\int_{\mathbb R^2}\bigg[\int_{0}^{1}{\check{M_3}}^{\epsilon}\left(y_1, y_2, s\right)ds\bigg]^2.
\end{eqnarray}
In the above, letting $\epsilon\rightarrow 0$ to reach the limit $\check{M_3}$ and $\kappa\rightarrow\infty$ to recover the estimate in the whole space $\mathbb R^3$, we have
\begin{eqnarray}
\limsup_{\eta\rightarrow 0}\int_{\mathbb R^3}\bigg[\left(\left(a_p\right)_{\eta}\right)^2+\left(\zeta-\left(M_3\right)_{\eta}\right)^2 + \left(a_{\eta}\right)^2\bigg]=0.
\end{eqnarray}
It then immediately follows, accounting for the previous weak converges \eqref{eq:a_p} and \eqref{eq:con_an}, that
\begin{eqnarray}
\left(a_p\right)_{\eta}\rightarrow 0 \mbox{ strongly in } L^2\left(\mathbb R^3\right),\\
a_\eta\rightarrow 0 \mbox{ strongly in } L^2\left(\mathbb R^3\right)
\end{eqnarray}
and $\zeta=\check{M_3}$. Hence the claims \eqref{eq:reduced_stray1}, \eqref{eq:reduced_stray2} follow, and as a consequence, the assertion \eqref{eq:str2'} holds.
\end{proof}
In the following, we improve the obtained convergences.
\begin{proposition}[Improved convergence]\label{improved_convergence}
Let $\left(\Qvec_{\eta}, \Mvec_{\eta}\right)$ be a sequence of minimizers of the functional $\mathcal{E}_{\eta}\left(\Qvec, \Mvec\right)$. Also, let $\Mvec_{\eta}\rightarrow\check{\Mvec}$ in $W^{1,2}\left(\Omega; \mathbb R^3\right)$ and $\Phi_{\eta}$ be the corresponding magnetostatic potential, i.e., the solution of \eqref{eq:reduced_maxwell} related to $\Mvec_{\eta}$. Then we have 
\begin{eqnarray}
\nabla_p\Phi_{\eta}\rightharpoonup 0 \mbox{ weakly in } W^{1,2}\left(\mathbb R^3\right),\label{eq:improved_con1}\\
\frac{1}{\eta}\Phi_{\eta,3}\rightharpoonup \zeta \mbox{ weakly in } W^{1,2}\left(\mathbb R^3\right).\label{eq:improved_con2}
\end{eqnarray}
\end{proposition}
\noindent Moreover, up to a subsequence, we have
\begin{eqnarray}
\nabla_p\Phi_{\eta}\rightarrow 0 \mbox{ strongly in } L^6\left(\mathbb R^3\right),\\
\frac{1}{\eta}\Phi_{\eta,3}\rightarrow \Tilde{M_3} \mbox{ strongly in } L^6\left(\mathbb R^3\right).  
\end{eqnarray}
\begin{proof}
Assume by \Cref{th:minimizer} that $\left(\Qvec_{\eta}, \Mvec_{\eta}\right)$ be the minimizers of the functional $\mathcal{E}_{\eta}\left(\Qvec, \Mvec\right)$. By selecting some fixed $\left(\Qvec, \Mvec\right)$ in the energy inequality, it similarly follows that 
\begin{eqnarray}
||\nabla_p\Mvec_{\eta}||_{L^2}\le D_1,\label{eq:prop1}\\
\bigg|\bigg|\frac{1}{\eta}\left(M_{3}\right)_{\eta, 3}\bigg|\bigg|_{L^2}\le D_2.
\end{eqnarray}
By using the elliptic regularity in \eqref{eq:reduced_maxwell}, we have
\begin{eqnarray}
\bigg|\bigg|\nabla_p\Phi_{\eta}+\frac{1}{\eta}\left(\Phi_{3}\right)_{\eta, 3}\bigg|\bigg|_{W^{1,2}}\le D_3\left(\bigg|\bigg|\nabla_p\cdot\left(\Mvec_p\right)_{\eta}+\frac{1}{\eta}\left(M_3\right)_{\eta,3}\bigg|\bigg|_{L^2}+\big|\big|\Phi_{\eta}\big|\big|_{L^2} \right)<\infty.\label{eq:improved_con3}
\end{eqnarray}
Therefore, up to a subsequence
\begin{eqnarray}
\nabla_p\Phi_{\eta}\rightharpoonup 0 \mbox{ weakly in } W^{1,2}\left(\mathbb R^3\right),\\
\frac{1}{\eta}\Phi_{\eta,3}\rightharpoonup \zeta \mbox{ weakly in } W^{1,2}\left(\mathbb R^3\right).
\end{eqnarray}
Again, due to the compact embedding  $W^{1,2}\hookrightarrow L^6$, up to a subsequence, it holds that
\begin{eqnarray}
\nabla_p\Phi_{\eta}\rightarrow 0 \mbox{ strongly in } L^6\left(\mathbb R^3\right),\\
\frac{1}{\eta}\Phi_{\eta,3}\rightarrow \zeta \mbox{ strongly in } L^6\left(\mathbb R^3\right).   
\end{eqnarray}
By the uniqueness of the $L^2$-limit of \eqref{eq:reduced_stray1} and \eqref{eq:reduced_stray2}, we infer that
\begin{eqnarray}
\nabla_p\Phi_{\eta}\rightarrow 0 \mbox{ strongly in } L^6\left(\mathbb R^3\right),\\
\frac{1}{\eta}\Phi_{\eta,3}\rightarrow \Tilde{M_3} \mbox{ strongly in } L^6\left(\mathbb R^3\right).       
\end{eqnarray}
\end{proof}
\begin{remark}
As it will become clear that the obtained improved convergences in the reminiscence of the scalar potential will nullify the nematic-stray field coupling contribution in the limiting ferronemtic energy (cf. \Cref{main_theorem}).
\end{remark}
\begin{remark}
We also mention the paper \cite{garcia2004one} for another approach based on Fourier transform methods to the dimension reduction of stray field energy.
\end{remark}
\begin{remark}
As we will finally observe in \Cref{main_theorem}, the contribution of nonlocal magnetostatic or stray field energy in the bulk ferronematic setting reduces to a local contribution accounted by the term $\int_{\omega}\frac{1}{2}M_3^2$. We speculate that other thin-film limits of micromagnetics may contribute differently in the reduced limit. The other local approximations are due to \cite{carbou2001thin, kohn2005another}, where the nonlocal magnetostatic energy reduces to a local contribution in the form of a boundary penalty $\int_{\partial\Omega}\left(\Mvec_p\cdot\bm{\nu}\right)^2$ (where $\bm{\nu}$ is the unit outward normal to the boundary $\partial\Omega$), while \cite{desimone2004recent} discusses singular magnetization patterns e.g.\ Néel wall, interior, boundary vortices, etc.. These thin-film regimes differ from each other by a logarithmic order. \newline
Also, there are regimes by Moser \cite{moser2004boundary, moser2003ginzburg, moser2005moving}, where nucleation of boundary vortices is mentioned with nonlocal vortex interaction. In contrast to those works, the thin film limit by Ignat and Kurzke \cite{ignat2023effective, ignat2021global} discusses the nucleation of boundary vortices, where the renormalized energy appears as a local contribution.
\subsection{Dimension reduction for the LdG order parameter $\Qvec$ and magnetization $\Mvec$}
\end{remark}
\begin{proposition}\label{reduced-QM}
Let $\left(\Qvec_{\eta}, \Mvec_{\eta}\right)$ be a sequence of minimizers of the functional $\mathcal{E}_{\eta}\left(\Qvec, \Mvec\right)$. Then the following convergences holds, up to a subsequence
\begin{eqnarray}
\nabla_p\Qvec_{\eta}&\rightarrow&\nabla_p\Qvec \mbox{ strongly in } L^2\left(\Omega; \mathbb R^{3\times 3}\right),\label{eq:strong_gQ}\\
\frac{1}{\eta}\Qvec_{\eta,3}&\rightarrow& 0 \mbox{ strongly in } L^2\left(\Omega; \mathbb R^{3\times 3}\right),\label{eq:strong1_gQ}\\
\nabla_p\Mvec_{\eta}&\rightarrow&\nabla_p\Mvec \mbox{ strongly in } L^2\left(\Omega; \mathbb R^{3}\right),\label{eq:strong_gM}\\
\frac{1}{\eta}\Mvec_{\eta,3}&\rightarrow& 0 \mbox{ strongly in } L^2\left(\Omega; \mathbb R^{3}\right)\label{eq:strong1_gM}.
\end{eqnarray}
\end{proposition}
\begin{proof}
Thanks to \Cref{th:minimizer}, we assume that $\left(\Qvec_{\eta}, \Mvec_{\eta}\right)$ be a sequence of minimizers of the energy functional $\mathcal{E}_{\eta}\left(\Qvec, \Mvec\right)$, yields
\begin{eqnarray}
\mathcal{E}_{\eta}\left(\Qvec_{\eta}, \Mvec_{\eta}\right)\le\mathcal{E}_{\eta}\left(\Qvec, \Mvec\right)<\infty\label{eq:bound}
\end{eqnarray}
for some fixed $\left(\Qvec, \Mvec\right)\in \mathcal{A}_{\eta}\times\mathcal{S}_{\eta}$. Therefore, we have
\begin{eqnarray}
||\nabla_p\Mvec_{\eta}||_{L^2}\le D_1, \bigg|\bigg|\frac{1}{\eta}\Mvec_{\eta,3}\bigg|\bigg|_{L^2}\le D_2.\label{eq:M_conv1}
\end{eqnarray}
Thus, up to a subsequence,
\begin{eqnarray}
\nabla_p\Mvec_{\eta}&\rightharpoonup&\nabla_p\Mvec \mbox{ weakly in } L^2,\label{eq:M_conv3}\\
\frac{1}{\eta}\Mvec_{\eta,3}&\rightharpoonup& L_1  \mbox{ weakly in } L^2.\label{eq:M_conv4}
\end{eqnarray}
Similarly,
\begin{eqnarray}
||\nabla_p\Qvec_{\eta}||_{L^2}\le D'_1, \bigg|\bigg|\frac{1}{\eta}\Qvec_{\eta,3}\bigg|\bigg|_{L^2}\le D'_2.\label{eq:Q_conv2}
\end{eqnarray}
Therefore, up to a subsequence, we have
\begin{eqnarray}
\nabla_p\Qvec_{\eta}&\rightharpoonup&\nabla_p\Qvec \mbox{ weakly in } L^2,\label{eq:Q_conv3}\\
\frac{1}{\eta}\Qvec_{\eta,3}&\rightharpoonup& L_2  \mbox{ weakly in } L^2.\label{eq:Q_conv4}
\end{eqnarray}
We select $\Qvec_{\eta}=\Qvec$ such that $f^0_{sur}\left(\Qvec, \hat{x}_3\right)=0$ or $\Qvec\hat{x}_3=\beta\hat{x}_3$, and $\Mvec_{\eta}=\Mvec$ such that $\Mvec_{,3}=0$ as a recovery sequence, and obtain
\begin{eqnarray}
\limsup_{\eta\rightarrow 0}\int_{\Omega}|\nabla_p\Qvec_{\eta}|^2+\frac{1}{\eta^2}\Qvec_{\eta,3}+|\nabla_p\Mvec_{\eta}|^2+\frac{1}{\eta^2}\Mvec_{\eta,3}\le\int_{\Omega}|\nabla_p\Qvec|^2+|\nabla_p\Mvec|^2.\label{lim-sup-eq}
\end{eqnarray}
Applying the lower semicontinuity property of the $L^2$-norm to convergences \eqref{eq:M_conv3}, \eqref{eq:M_conv4}, \eqref{eq:Q_conv3} and \eqref{eq:Q_conv4}, and using the limsup inequality \eqref{lim-sup-eq}, we have
\begin{eqnarray}
& &\int_{\Omega}|\nabla_p\Qvec|^2+|\nabla_p\Mvec|^2+L^2_1+L_2^2\nonumber\\
&\le&\liminf_{\eta\rightarrow 0}\int_{\Omega}|\nabla_p\Qvec_{\eta}|^2+\frac{1}{\eta^2}\Qvec_{\eta,3}+|\nabla_p\Mvec_{\eta}|^2+\frac{1}{\eta^2}\Mvec_{\eta,3}\nonumber\\
&\le&\limsup_{\eta\rightarrow 0}\int_{\Omega}|\nabla_p\Qvec_{\eta}|^2+\frac{1}{\eta^2}\Qvec_{\eta,3}+|\nabla_p\Mvec_{\eta}|^2+\frac{1}{\eta^2}\Mvec_{\eta,3}\nonumber\\
&\underset{\eqref{lim-sup-eq}}{\le}&\int_{\Omega}|\nabla_p\Qvec|^2+|\nabla_p\Mvec|^2.\label{eq:limit_ineq}
\end{eqnarray}
By \eqref{eq:limit_ineq}, $L_1=L_2=0$ and
\begin{eqnarray}
\lim_{\eta\rightarrow 0}\int_{\Omega}|\nabla_p\Qvec_{\eta}|^2+\frac{1}{\eta^2}\Qvec_{\eta,3}+|\nabla_p\Mvec_{\eta}|^2+\frac{1}{\eta^2}\Mvec_{\eta,3}=\int_{\Omega}|\nabla_p\Qvec|^2+|\nabla_p\Mvec|^2.
\end{eqnarray} 
The strong convergences \eqref{eq:strong_gQ}, \eqref{eq:strong1_gQ}, \eqref{eq:strong_gM} and \eqref{eq:strong1_gM} follow by repeating the same approach as in Proposition \ref{energy_bound}.
\end{proof}
\begin{remark}
Here we note that due to a priori assumption of nematic behavior on the boundary, i.e., the special regime of the surface energy, the limiting nematic order parameter $\Qvec$ has two independent variables \cite{golovaty2014minimizers}.
\end{remark}
Finally, we are in a position to present the limiting energy functional.
\begin{theorem}\label{main_theorem}
Let $\Omega\subset\mathbb R^3$ be a bounded domain such that $\partial\Omega$ is Lipschitz and $\left(\Qvec_{\eta}, \Mvec_{\eta}\right)$ be a sequence of minimizers of $\mathcal{E}_{\eta}\left(\Qvec, \Mvec\right)$. Then, up to a subsequence 
\begin{eqnarray}
\Qvec_{\eta}\rightharpoonup\Qvec &\mbox{ weakly in }& W^{1,2}\left(\Omega; \mathbb R^{3\times 3}\right),\label{eq:Q_1}\\
\Mvec_{\eta}\rightharpoonup\Mvec &\mbox{ weakly in }& W^{1,2}\left(\Omega; \mathbb R^3\right)\label{eq:M_1}
\end{eqnarray}
and $\left(\Qvec, \Mvec\right)$ is a minimizer of the limiting energy functional $\mathcal{E}_0\left(\Qvec, \Mvec\right)$. The limiting energy functional $\mathcal{E}_0\left(\Qvec, \Mvec\right)$ is obtained as
\begin{eqnarray}
\mathcal{E}_0\left(\Qvec, \Mvec\right)=\int_{\omega}\bigg[\frac{K_1}{2}|\nabla_p\Qvec|^2+\frac{K_2}{2}|\nabla_p\Mvec|^2+f_B\left(\Qvec, \Mvec\right)-\mu_0\Mvec\cdot\Hvec_{ext}\nonumber\\
~~~~~~~+\frac{\mu_0}{2}M_3^2 - \frac{\chi_1\mu_0}{2}\Qvec\Hvec_{ext}\cdot\Hvec_{ext}\bigg]. \label{eq:energy_reduced}
\end{eqnarray}
\end{theorem} 
\begin{proof}
The claims \eqref{eq:Q_1} and \eqref{eq:M_1} are immediate consequences of \eqref{eq:bound}. By the compact embedding $W^{1,2}\hookrightarrow L^6$, up to a subsequence, we have
\begin{eqnarray}
\Qvec_{\eta}\rightarrow\Qvec \mbox{ strongly in } L^6\left(\Omega; \mathbb R^{3\times 3}\right),\label{strng_Q}\\
\Mvec_{\eta}\rightarrow\Mvec \mbox{ strongly in } L^6\left(\Omega; \mathbb R^{3}\right).\label{strng_M}
\end{eqnarray}
Due to Proposition \ref{improved_convergence} and the convergence \eqref{strng_Q}, up to subsequences (without reindexing), it follows that
\begin{eqnarray}
\Qvec_{\eta}\left(\nabla_p\Phi_{\eta}+\frac{1}{\eta}\left(\Phi_{,3}\right)_{\eta}e_3\right)\cdot\left(\nabla_p\Phi_{\eta}+\frac{1}{\eta}\left(\Phi_{,3}\right)_{\eta}e_3\right)\rightarrow 0 \mbox{ in } L^1 \mbox{ as } \eta\rightarrow 0.\label{eq:coupling_str}
\end{eqnarray}
By the closure property of strong convergences \eqref{strng_Q} and \eqref{strng_M}, up to subsequences (without reindexing)
\begin{eqnarray}
f\left(\Qvec_{\eta}, \Mvec_{\eta}\right)\rightarrow f\left(\Qvec, \Mvec\right) \mbox{ in } L^1 \mbox{ as } \eta\rightarrow 0.\label{conv_bulk}
\end{eqnarray}
Using Proposition \ref{improved_convergence}, Proposition \ref{reduced-QM} along with the convergences \eqref{eq:coupling_str} and \eqref{conv_bulk}, we obtain
\begin{eqnarray}
& &\lim_{\eta\rightarrow 0}\mathcal{E}_{\eta}\left(\Qvec_{\eta}, \Mvec_{\eta}\right)\nonumber\\
&=&\lim_{\eta\rightarrow 0}\int_{\Omega}\biggr[\frac{K_1}{2}\left(|\nabla_p\Qvec_{\eta}|^2+\frac{1}{\eta^2}\Qvec_{\eta,3}\right)+\frac{K_2}{2}\left(|\nabla_p\Mvec_{\eta}|^2+\frac{1}{\eta^2}\Mvec_{\eta,3}\right)+f\left(\Qvec_{\eta}, \Mvec_{\eta}\right)-\mu_0\Mvec_{\eta}\cdot\Hvec_{ext}\nonumber\\
& &~~~~~~~~~~~~~~- \frac{\chi_1\mu_0}{2}\Qvec_{\eta}\Hvec_{ext}\cdot\Hvec_{ext}- \frac{\chi_1\mu_0}{2}\Qvec_{\eta}\left(\nabla_p\Phi_{\eta}+\frac{1}{\eta}\left(\Phi_{,3}\right)_{\eta}e_3\right)\cdot\left(\nabla_p\Phi_{\eta}+\frac{1}{\eta}\left(\Phi_{,3}\right)_{\eta}e_3\right)\biggr]\nonumber\\
&=&\int_{\Omega}\bigg[\frac{K_1}{2}|\nabla_p\Qvec|^2+\frac{K_2}{2}|\nabla_p\Mvec|^2+\frac{\mu_0}{2}M_3^2-\mu_0\Mvec\cdot\Hvec_{ext}-\frac{\chi_1\mu_0}{2}\Qvec\Hvec_{ext}\cdot\Hvec_{ext}\bigg].
\end{eqnarray}
Therefore the last claim of \Cref{main_theorem} follows by comparing $\mathcal{E}_{\eta}\left(\Qvec_{\eta}, \Mvec_{\eta}\right)$ to $\mathcal{E}_{\eta}\left(\Qvec, \Mvec\right)$ for any fixed $\left(\Qvec, \Mvec\right)\in\mathcal{Q}_{\eta}\times\mathcal{M}_{\eta}$. By the choice of recovery sequence in Proposition \ref{reduced-QM}, the limiting ferronematic energy (cf. \eqref{eq:energy_reduced}) is defined over $\omega$.
\end{proof}
\section{Numerical observations}\label{sec:numerical observations}
In the obtained limiting ferronematic energy (cf. \eqref{eq:energy_reduced}), applying suitable scaling (see  \cite{dutta2025study}) to make it dimensionless, we can write the associated gradient flow system as follows (for details, see \cite{dutta2026existence})
\begin{eqnarray}
\small
\left\{ \begin{aligned}
\tau_1\frac{\partial}{\partial t}\begin{pmatrix}
Q_{11}\\
Q_{12}
\end{pmatrix} &=2l_1\Delta \begin{pmatrix}
Q_{11}\\
Q_{12}
\end{pmatrix} - \left(\frac{|\Qvec|^2}{2} - 1\right)\begin{pmatrix}
Q_{11}\\
Q_{12}
\end{pmatrix} + \frac{c_1}{2}\begin{pmatrix}
M_1^2 - M_2^2\\
2M_1M_2
\end{pmatrix}+ \frac{c_2}{2}\begin{pmatrix}
H_1^2 - H_2^2\\
2H_1H_2
\end{pmatrix},\\
\tau_2\frac{\partial}{\partial t}\begin{pmatrix}
M_1\\
M_2\\
M_3
\end{pmatrix}&= \xi l_2\Delta \begin{pmatrix}
M_1\\
M_2\\
M_3
\end{pmatrix} - \xi\left(|\Mvec|^2-1\right)\begin{pmatrix}
M_1\\
M_2\\
M_3
\end{pmatrix} + c_1\begin{pmatrix}
Q_{11}M_1 + Q_{12}M_2\\
Q_{12}M_1 - Q_{11}M_2\\
0
\end{pmatrix}
+ c_3\xi\begin{pmatrix}
H_1\\
H_2\\
H_3-M_3
\end{pmatrix},
\end{aligned}\right.\nonumber\\ \label{eq:26}
\end{eqnarray}
where $\tau_1$ and $\tau_2$ are dissipation coefficients \cite{onsager1931reciprocal1, onsager1931reciprocal2}. We prescribe a tangent-type boundary condition with $M_3=0$ on the boundary and solve \eqref{eq:26} using the Crank-Nicolson finite difference method with Newton's linearization technique \cite{Burden-Faires, smith1985numerical}. For the particular solution scheme, we refer to \cite{dutta2025study}, and for more details, see \cite{dutta2026existence}. In the limits of dissipation coefficients $\tau_1, \tau_2\rightarrow 0$, the gradient flow system (cf. \eqref{eq:26}) converges to the corresponding Euler-Lagrange system, and we present some numerical observations for the nonzero $M_3$ over the case when $M_3=0$. Physically, the term nonzero $M_3$ is responsible for a stray field due to the Gioia and James approximation \cite{Giogia}.
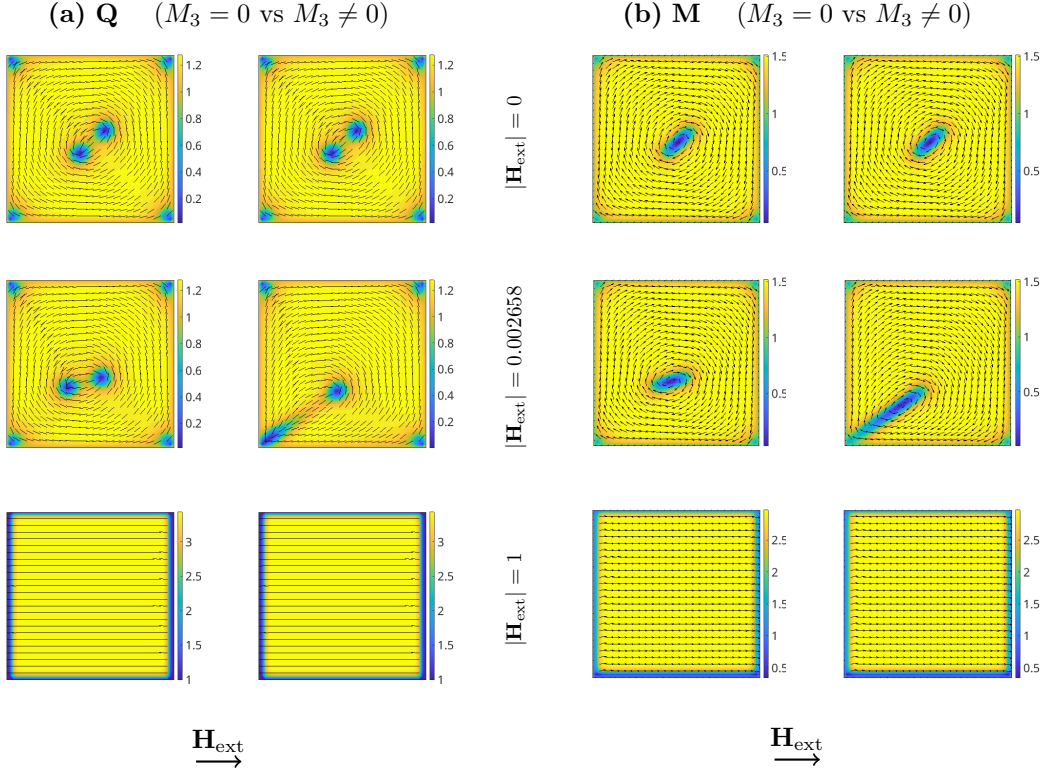
\begin{figure}[ht!]
\centering

\begin{subfigure}[t]{0.49\textwidth}
\centering
\textbf{(a)} $\Qvec$ \quad ($M_3=0$ vs $M_3\neq0$)

\begin{tikzpicture}
  \node (Qbase) {\includegraphics[width=3.6cm]{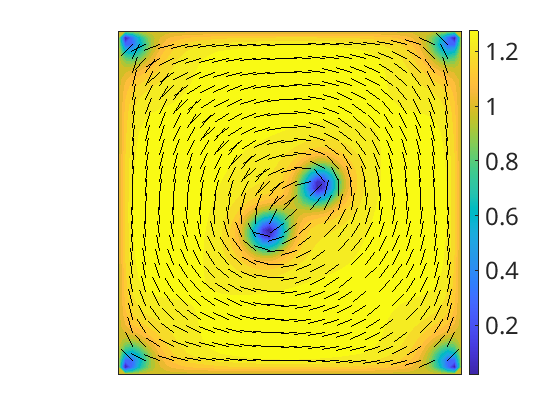}};
  \node[anchor=west] at ([xshift=-0.52cm] Qbase.east)
    {\includegraphics[width=3.6cm]{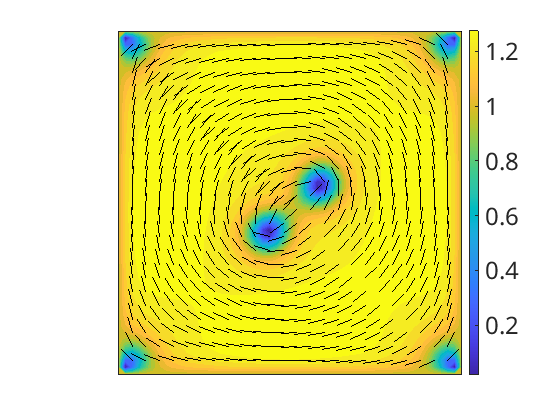}};
    \node[anchor=west] at ([xshift=35mm] Qbase.east)
    {\scalebox{0.8}{\rotatebox{90}{$|\Hvec_{\mathrm{ext}}|=0$}}};
\end{tikzpicture}
\vspace{1mm}
\begin{tikzpicture}
  \node (Qbase) {\includegraphics[width=3.6cm]{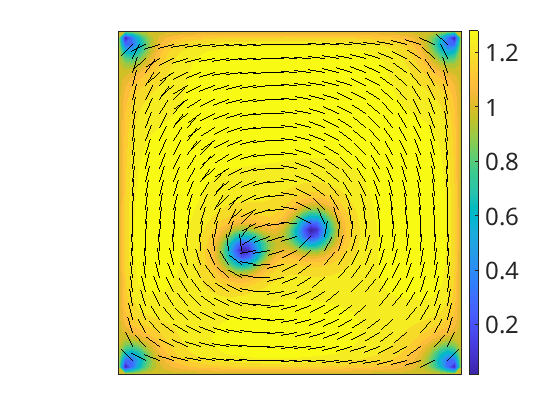}};
  \node[anchor=west] at ([xshift=-0.52cm] Qbase.east)
    {\includegraphics[width=3.6cm]{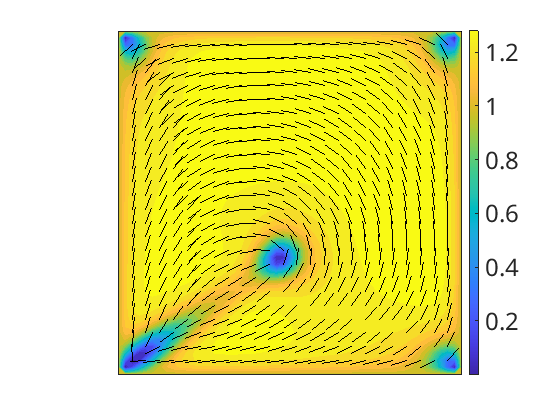}};
    \node[anchor=west] at ([xshift=35mm] Qbase.east)
    {\scalebox{0.8}{\rotatebox{90}{$|\Hvec_{\mathrm{ext}}|=0.002658$}}};
\end{tikzpicture}
\vspace{1mm}
\begin{tikzpicture}
  \node (Qbase) {\includegraphics[width=3.6cm]{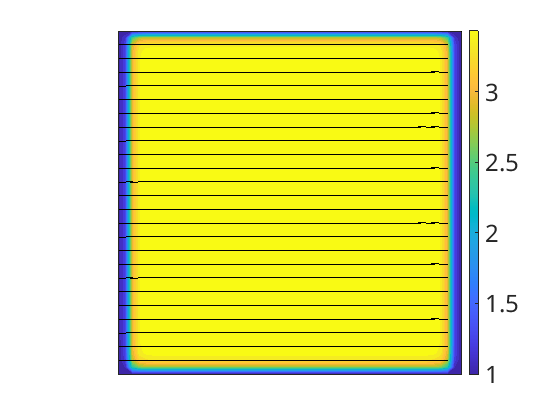}};
  \node[anchor=west] at ([xshift=-0.52cm] Qbase.east)
    {\includegraphics[width=3.6cm]{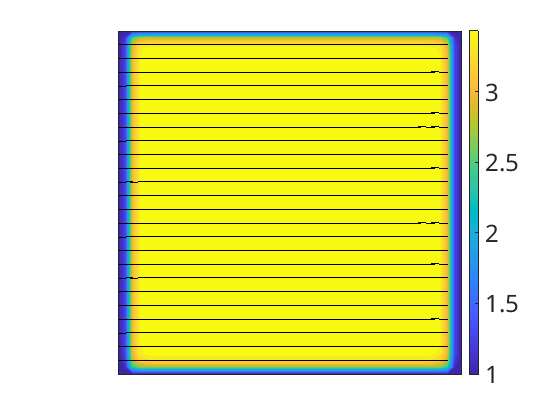}};
    \node[anchor=west] at ([xshift=35mm] Qbase.east)
    {\scalebox{0.8}{\rotatebox{90}{$|\Hvec_{\mathrm{ext}}|=1$}}};
\end{tikzpicture}

\begin{tikzpicture}
  \draw[->, thick] (0,0) -- (0.6,0)
    node[midway, above] {$\Hvec_{\mathrm{ext}}$};
\end{tikzpicture}
\end{subfigure}
\hfill
\begin{subfigure}[t]{0.49\textwidth}
\centering
\textbf{(b)} $\Mvec$ \quad ($M_3=0$ vs $M_3\neq0$)

\begin{tikzpicture}
  \node (Mbase) {\includegraphics[width=3.6cm]{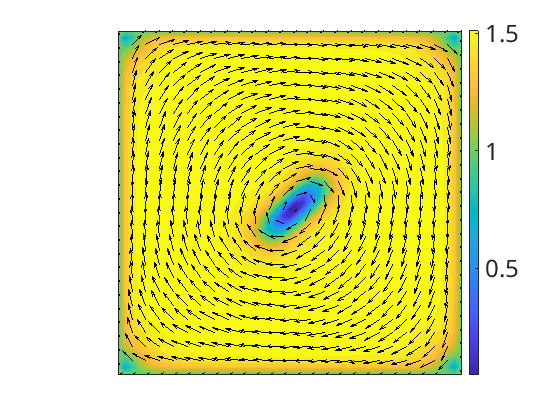}};
  \node[anchor=west] at ([xshift=-0.52cm] Mbase.east)
    {\includegraphics[width=3.6cm]{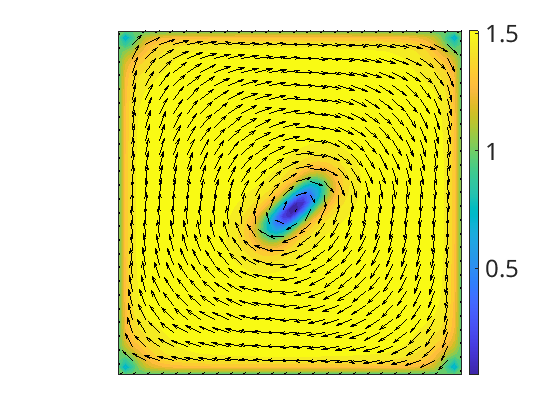}};
\end{tikzpicture}
\vspace{1mm}
\begin{tikzpicture}
  \node (Mbase) {\includegraphics[width=3.6cm]{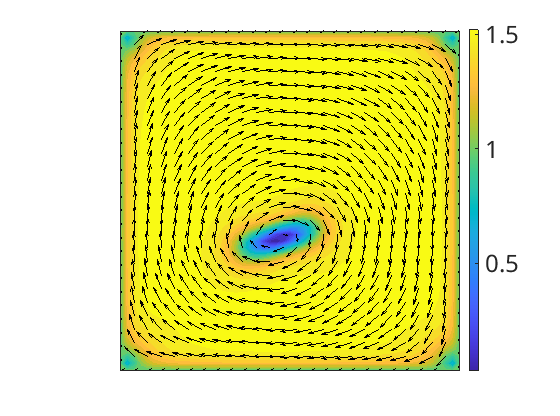}};
  \node[anchor=west] at ([xshift=-0.52cm] Mbase.east)
    {\includegraphics[width=3.6cm]{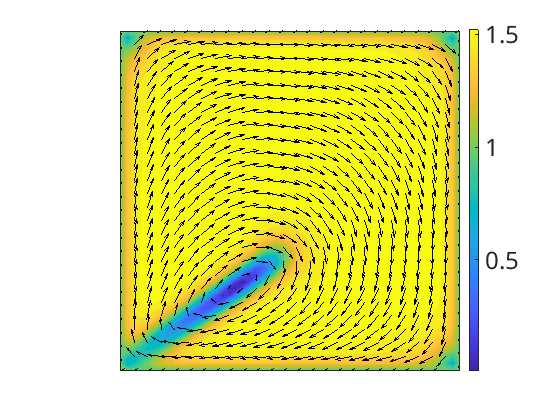}};
\end{tikzpicture}
\vspace{1mm}
\begin{tikzpicture}
  \node (Mbase) {\includegraphics[width=3.6cm]{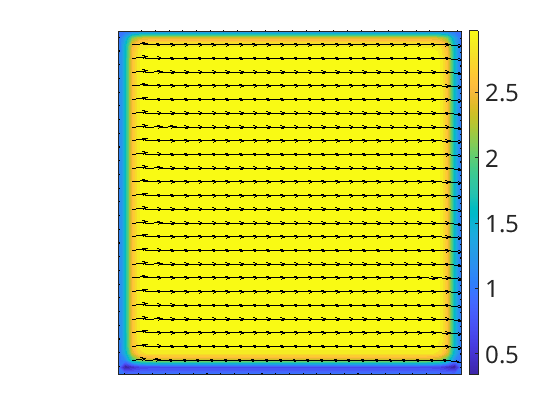}};
  \node[anchor=west] at ([xshift=-0.52cm] Mbase.east)
    {\includegraphics[width=3.6cm]{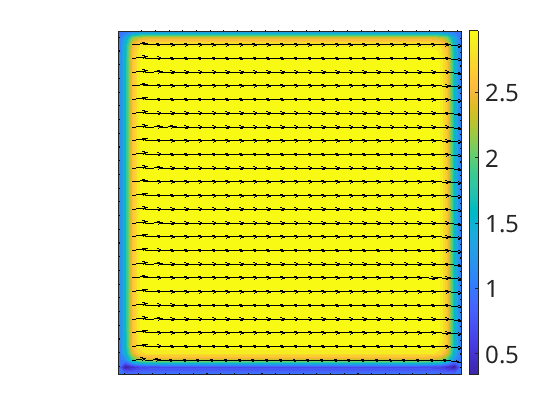}};
\end{tikzpicture}
\vspace{1mm}
\begin{tikzpicture}
  \draw[->, thick] (0,0) -- (0.6,0)
    node[midway, above] {$\Hvec_{\mathrm{ext}}$};
\end{tikzpicture}
\end{subfigure}
\caption{\textbf{Nematic and magnetic profiles for different activated magnetic fields $\Hvec_{ext}$ ($0, 0.002658, 1$), with ($M_3\neq0$) and without ($M_3=0$) stray field energy. Values of parameters: $l_1=l_2=0.001$, $c_1=0.5$, $c_2=8$, $c_3=2$, $\xi=1$, $\tau_1=\tau_2=0.0001$. 
}}
\label{Fig-new1}
\end{figure} 
\\
\\
The observations we are presenting \Cref{Fig-new1} are discussed in \cite{dutta2025study, dutta2026existence}. Here we present the case that the stray field is localized in the interior of the domain by imposing $M_3=0$ on the boundary. Our purpose is to show a more visible importance of the limiting ferronematic energy \eqref{eq:energy_reduced} in a physical setting.
In \Cref{Fig-new1}, we have presented a set of ferronematic profiles, where the middle row indicates a clean difference in the defect localization in the presence of an activated magnetic field $\Hvec_{ext}=\left(0.002658, 0, 0\right)$. This observation supports the physical significance of the obtained reduced ferronematic energy \eqref{eq:energy_reduced}. The numerical observations taken with local magnetostatic energy in a thin film setting clearly support the proposed generalized ferronematic energy (cf. \eqref{energy-intro}) in the bulk setting, too.
\section{Conclusions}\label{conclusions}
This work can be a first step toward understanding effective energy laws for ferroenematic materials via mathematical explanations. We propose a ferronematic energy in a bulk setting following \cite{Mertelj}, where we explicitly incorporate the nonlocal influence of the magnetostatic energy. We then provide a reduced ferronematic energy via the weak convergence methods embedded in $\Gamma$-convergence in a thin-film setting. The approach follows the thin-film limit derivation for the magnetostatic energy by Gioia and James \cite{Giogia} in the ferronematics setup. There might be scopes to explore other finer thin-film limits (which generally differ by a logarithmic scale) of magnetostatic energy (see e.g.\ \cite{desimone2004recent}) in the ferronematics framework.  
\section{Acknowledgements}
S.D. is grateful to DAAD (Deutscher Akademischer Austauschdienst) for supporting this work through a doctoral fellowship, without which this work would not have been possible. Also, S.D. is grateful to Stefanie Petermichl for the support of her research through the Humboldt Professorship award from Humboldt foundation. S.D. would like to thank Anja Schlömerkemper for early discussions about the general setting of this work.

\bibliographystyle{unsrtnat}
\bibliography{references}

\end{document}